\documentclass[reqno]{amsart}

\usepackage{amsmath}
\usepackage{amssymb}
\usepackage{mathdots}
\usepackage[bb=boondox]{mathalfa}

\usepackage{tikz}

\newif\ifstartedinmathmode
\newcommand\encircled[1]{%
  \relax\ifmmode\startedinmathmodetrue\else\startedinmathmodefalse\fi%
  \tikz[baseline,anchor=base]{%
  \node[draw=red,circle,outer sep=0pt,inner sep=.2ex]
    {\ifstartedinmathmode$#1$\else#1\fi};}%
}
\theoremstyle{plain}
\newtheorem{theorem}{Theorem}[section]
\newtheorem{lemma}[theorem]{Lemma}
\newtheorem{corollary}[theorem]{Corollary}
\newtheorem{proposition}[theorem]{Proposition}

\theoremstyle{definition}
\newtheorem{remark}[theorem]{Remark}

\newtheorem{example}[theorem]{Example}
\newtheorem{question}[theorem]{Question}

\numberwithin{equation}{section}

\allowdisplaybreaks[1]

\newcommand{\BC}{{\mathbb C}}
\newcommand{\BF}{{\mathbb F}}
\newcommand{\BH}{{\mathbb H}}

\newcommand{\BL}{{\mathbb L}}

\newcommand{\BR}{{\mathbb R}}

\newcommand{\cE}{{\mathcal E}}
\newcommand{\cH}{{\mathcal H}}

\newcommand{\cL}{{\mathcal L}}

\newcommand{\cS}{{\mathcal S}}

\newcommand{\fC}{{\mathfrak C}}
\newcommand{\fE}{{\mathfrak E}}

\newcommand{\fS}{{\mathfrak S}}

\newcommand{\fY}{{\mathfrak Y}}

\newcommand{\wtilL}{\widetilde{L}}

\newcommand{\whatA}{\widehat{A}}

\newcommand{\al}{\alpha}
\newcommand{\be}{\beta}

\newcommand{\de}{\delta}

\newcommand{\la}{\lambda}

\newcommand{\rank}{\textup{rank\,}}

\newcommand{\im}{\textup{Im}\,}

\newcommand{\kr}{\textup{Ker\,}}
\newcommand{\diag}{\textup{diag}}

\newcommand{\mat}[1]{\begin{bmatrix} #1 \end{bmatrix}}
\newcommand{\sbm}[1]{\left[\begin{smallmatrix} #1\end{smallmatrix}\right]}
\newcommand{\ov}[1]{{\overline{#1}}}

\newcommand{\tu}[1]{\textup{#1}}
\newcommand{\wtil}[1]{{\widetilde{#1}}}
\newcommand{\what}[1]{{\widehat{#1}}}

\newcommand{\ands}{\quad\mbox{and}\quad}

\newcommand{\trace}{\textup{trace\,}}

\newcommand{\BBone}{\mathbb{1}}
\newcommand{\OneVec}{\vec{\mathbf{1}}}

\newcommand{\vect}{\operatorname{vec}}

\begin{document}

\title[Matrix maps for which positivity and complete positivity coincide]{Linear matrix maps for which positivity and complete positivity coincide}

\author[S. ter Horst]{S. ter Horst}
\address{S. ter Horst, Department of Mathematics, Research Focus Area:\ Pure and Applied Analytics, North-West
University, Potchefstroom, 2531 South Africa and DSI-NRF Centre of Excellence in Mathematical and Statistical Sciences (CoE-MaSS)}
\email{Sanne.TerHorst@nwu.ac.za}

\author[A. Naud\'{e}]{A. Naud\'{e}}
\address{A. Naud\'{e}, Faculty of Engineering and the Built Environment, Academic Development Unit, University of the Witwatersrand, Johannesburg, 2000 South Africa and DSI-NRF Centre of Excellence in Mathematical and Statistical Sciences (CoE-MaSS)}
\email{naudealma@gmail.com}

\thanks{This work is based on the research supported in part by the National Research Foundation of South Africa (Grant Number 118513 and 127364).}





\begin{abstract}
By the Choi matrix criteria it is easy to determine if a specific linear matrix map is completely positive, but to establish whether a linear matrix map is positive is much less straightforward. In this paper we consider classes of linear matrix maps, determined by structural conditions on an associated matrix, for which positivity and complete positivity coincide. The basis of our proofs lies in a representation of $*$-linear matrix maps going back to work of R.D. Hill which enables us to formulate a sufficient condition in terms of surjectivity of certain bilinear maps.
\end{abstract}

\maketitle

\section{Introduction}

Positive and completely positive matrix maps play a profound role in many fields of mathematics as well as in mathematical physics, cf., \cite{S13,SS05,L75,KMcCSZ19,P19Arx,P02,TT88,O91} and references given there. While the structure of the class of completely positive matrix maps is much better understood and it is easy to verify complete positivity via the Choi matrix criteria, the class of positive matrix maps is much less studied and more intricate, in part because of the existence of non-decomposable positive matrix maps, i.e., maps that are not the sum of a completely positive and a completely co-positive map \cite{M12,ZC13,TT88}. Moreover, recent work shows that there are many more positive maps than completely positive maps \cite{KMcCSZ19}, using the connection with multivariable polynomials that are positive but not sums-of-squares, while other recent work \cite{B20Arx} focuses on explicit construction of positive maps that are not completely positive. In the present paper we take a different approach and focus on classes of linear maps for which positivity and complete positivity coincide. As indicated, via the Choi matrix it is not difficult to determine whether a specific linear map is completely positive. However, we want to consider this question independent of specific maps, and determine classes of linear maps, determined by certain structural properties, where this always happens. That this phenomenon occurs for certain classes of matrix maps came out of our study of certain interpolation problems, on which we will report in a separate publication, where the solution criteria is to show that a certain matrix map determined by the interpolation data is positive, but it turns out that positivity in that case coincides with complete positivity, independently of the data, so that the Choi matrix criteria is not only sufficient for the existence of a solution, but also necessary.

Throughout this paper $\BF=\BC$ or $\BF=\BR$. To avoid confusion about transposes or adjoints, symmetric and Hermitian matrices, etc., we shall use notation  as if $\BF=\BC$. We consider a linear matrix map
\begin{equation}\label{cL-Intro}
\cL:\BF^{q\times q} \to \BF^{n\times n}.
\end{equation}
 With $\cL$ we associate two matrices, the Choi matrix $\BL$ given by
\begin{equation}\label{Choi}
\BL=\left[\BL_{ij}\right] \in \BF^{nq \times nq},\ \  \BL_{ij}=\cL\left(\mathcal{E}_{ij}^{(q)}\right) \in\BF^{n \times n}\ \ \mbox{for $i=1,\ldots,n,\, j=1,\ldots,q$},
\end{equation}
where $\mathcal{E}_{ij}^{(q)}$ is the standard basis element in $\BF^{q \times q}$ with a 1 on position $(i,j)$ and zeros elsewhere, and what we call the matricization of $\cL$, which is the matrix $L\in \BF^{n^2 \times q^2}$ determined by the linear map
\begin{equation}\label{Matricization}
L:\BF^{q^2} \to \BF^{n^2},\quad L\,\left(\vect_{q}(V)\right)= \vect_{n}\left(\cL (V)\right),\quad V\in\BF^{q\times q},
\end{equation}
where $\vect_{r \times s} :\BF^{r \times s} \to \BF^{rs}$ is the vectorization operator, abbreviated to $\vect_r$ in case $r=s$. Decompose $L$ as a block matrix
\begin{equation}\label{Lblock}
L=\left[L_{ij}\right]\quad \mbox{with}\quad L_{ij}\in\BF^{n \times q}\quad\mbox{for $i=1,\ldots,n,\, j=1,\ldots,q$}.
\end{equation}
Then $\BL$ and $L$ determine each other in the following way: The $((j-1)q + i)$-th column of $L$ is given by $\vect_n(\BL_{ij})$ and the $((j-1)n + i)$-th column in $\BL$ is given by $\vect_{n \times q}(L_{ij})$. This relation between $\BL$ and $L$ corresponds to a matrix reordering that appeared in \cite{PH81}, was further studied in \cite{OH85} and rediscovered recently in \cite{P19}; see Proposition 4.1 and Section 3 in \cite{tHN2} for further details.\label{LBLrel}

The well-known Choi matrix criteria tells us that $\cL$ is completely positive if and only if $\BL$ is positive semidefinite. Following \cite{KMcCSZ19}, we say that $\cL$ is $*$-linear if $\cL(V^*)=\cL(V)^*$ for all $V\in\BF^{q\times q}$ (for $\BF=\BC$ this corresponds with $\cL$ mapping Hermitian matrices to Hermitian matrices), and it turns out that $*$-linearity is equivalent to $\BL$ being Hermitian. In particular, all completely positive linear matrix maps are $*$-linear. Positivity of $\cL$ corresponds to
\begin{equation}\label{PosCon}
\left(z\otimes x\right)^* \BL \left(z\otimes x\right) \ge 0 \quad \mbox{for all}\quad x\in\BF^n \text{ and } z \in \BF^q.
\end{equation}
In this paper we determine structural properties of $\cL$, formulated in terms of the matricization $L$, under which positivity of $\cL$, that is, \eqref{PosCon}, is sufficient to conclude that $\BL$ is positive semidefinite, and hence $\cL$ completely positive, which are independent of the precise linear map $\cL$ (i.e., the entries in $L$).

Set $m=\rank \BL$. The relation between $\BL$ and $L$ explained above implies that $m$ corresponds to the maximum number of linearly independent matrices among the block entries of $L$, that is,
\[
m=\dim \tu{span}\{ L_{ij} \colon i=1,\ldots,n, \, j=1,...,q \}\subset\BF^{n\times q}.
\]
In particular, it is possible to select $m$ linearly independent block entries of $L$. Our main result states that if one can choose these linearly independent matrices to satisfy a certain condition and all other matrices are equal, then positivity and complete positivity of $\cL$ will coinicde.

\begin{theorem}\label{T:Main}
Let $\cL$ in \eqref{cL-Intro} be a $*$-linear map with matricization $L=\left[L_{ij}\right]$ and Choi matrix $\BL$. Set $m=\rank \BL$. Assume one can choose linearly independent $L_1,\ldots,L_m\in\BF^{n \times q}$ among the block entries of $L$, say $L_k=L_{i_kj_k}$ for $k=1,\ldots m$, in such a way that:
\begin{equation}\label{C1intro}
\mbox{(C1)\ \  For each $k$ we have $j_l\neq j_k$ for each $l\neq k$ or $i_l\neq i_k$ for each $l\neq k$.}
\end{equation}
Assume further that all other block entries of $L$ are equal to a single matrix $L_0$ in the span of $L_1,\ldots,L_m$. Then $\cL$ is completely positive if and only if $\cL$ is positive.
\end{theorem}

This theorem will be proved in Section \ref{S:PtoCP}.

The additional condition that all other matrices must be equal to a single matrix $L_0$ is unfortunate and we do not know whether it can be removed. However, in a few specific examples where (C1) holds and more than a single matrix occurs among the remaining matrices we could still prove the result, while we have not been able to produce a counterexample. Specifically, when the independent matrices $L_1,\ldots,L_m$ can be chosen so that they are all in a single block column or in a single block row, then the result remains valid without assumptions on the remaining block matrices in $L$, see Proposition \ref{P:IndepsRowColumn} below.

\begin{question}\label{Q:OpenQ}
Does Theorem \ref{T:Main} remain true if the condition that the remaining block entries $L_{ij}$ for $(i,j)\neq (i_k,j_k)$, $k=1,\ldots,m$ should be equal to a single matrix $L_0$ is removed?
\end{question}

A specific subclass appears when one restricts to $L_0=0$, that is, when the non-zero block entries of $L$ form a linearly independent set. For this case the result is proved separately in Subsection \ref{SubS:L_0=0} as a stepping stone to the proof of Theorem \ref{T:Main}. In this case condition (C1) corresponds to a condition on the zero-pattern in the block matrix $L$. Theorem \ref{T:Main} only provides a necessary condition for positive maps $\cL$ to be completely positive. In the case where $L_0=0$ we give an example (in Example \ref{E:2x2upper} below) where the phenomenon still occurs, but (C1) does not hold. If matrices other than $L_0=0$ are allowed, this need not happen, at least for $\BF=\BR$, as follows from Example \ref{E:2x2Toeplitz} below.

The basis for our proof of Theorem \ref{T:Main} is a type of representation of linear matrix maps studied by R.D. Hill in \cite{H69,H73}, which we will refer to as Hill representations. These representations have the form
\begin{equation}\label{HillRepIntro}
\cL(V)=\sum_{k,l=1}^m \BH_{kl}\, A_k V A_l^*,\quad V\in\BF^{q \times q},
\end{equation}
for matrices $A_1,\ldots,A_m \in\BF^{n \times q}$. The matrix $\BH=\left[\BH_{kl}\right]_{k,l=1}^m\in\BF^{m\times m}$ is called the Hill matrix associated with the representation \eqref{HillRepIntro}. Moreover, we call the Hill representation \eqref{HillRepIntro} of $\cL$ minimal if $m$ is the smallest number of matrices $A_k$ that occurs in Hill representations for $\cL$, and it turns out that this smallest number equals $\rank \BL$. A detailed analysis of minimal Hill representations was conducted in \cite{tHN2} and a review of the relevant results from \cite{tHN2} will be given in Section \ref{S:Hill}.

If $\cL$ is given by the minimal Hill representation \eqref{HillRepIntro}, then the Choi matrix factors as
\[
\BL=\whatA^* \BH \whatA\quad \mbox{with } \widehat{A}^*:=\begin{bmatrix} \vect_{n \times q}\left(A_1\right) & \hdots & \vect_{n \times q}\left(A_m\right)  \end{bmatrix}\in \BF^{nq \times m}
\]
and $\whatA$ is a matrix with full row rank. As proved by Poluikis and Hill in \cite{PH81}, it follows that complete positivity of $\cL$ corresponds to positive definiteness of $\BH$. Using our criteria for positivity in \eqref{PosCon} and the above factorisation of $\BL$, it follows that a positive map $\cL$ is completely positive when the bilinear map determined by the matrix $\whatA$:
\begin{equation}\label{Bilinear}
(z,x)\mapsto \whatA (z\otimes x),\quad x\in\BF^n,\, z \in \BF^q
\end{equation}
is surjective. In fact, surjectivity of this bilinear map is independent of the choice of the Hill representation. To prove Theorem \ref{T:Main} we show that under the conditions in the theorem the bilinear map \eqref{Bilinear} is surjective independently of the choice of the linearly independent  matrices $L_1,\ldots,L_m$ and the matrix $L_0$, which works for $\BF=\BC$ and almost for $\BF=\BR$. In fact, for $\BF=\BR$ with $L$ as in Theorem \ref{T:Main}, the only case where the bilinear map \eqref{Bilinear} is not surjective is when $L$ is of the form
\[
L=\mat{
L_0&\cdots&L_0&L_q&L_0&\cdots&L_0\\
\vdots&&\vdots&\vdots&\vdots&&\vdots\\
L_0&\cdots&L_0&L_{q+r-2}&L_0&\cdots&L_0\\
L_1&\cdots&L_{s-1}&L_0&L_{s}&\cdots&L_{q-1}\\
L_0&\cdots&L_0&L_{q+r-1}&L_0&\cdots&L_0\\
\vdots&&\vdots&\vdots&\vdots&&\vdots\\
L_0&\cdots&L_0&L_{q+n-2}&L_0&\cdots&L_0
}
\]
with $L_0=\sum_{k=1}^{q+n-2}\al_k L_k$ such that $4\left(\sum_{k=1}^{q-1}\al_k\right)\left(\sum_{k=q}^{q+n-2}\al_k\right)>1$. In this case, however, positivity and complete positivity still coincides, which is proved by explicitly computing the range of the bilinear map \eqref{Bilinear}.

Little appears to be known about the ranges of bilinear maps, and we suspect that to resolve Question \ref{Q:OpenQ} a further study of surjectivity of bilinear maps might be required. However, we also point out that surjectivity of the associated bilinear maps \eqref{Bilinear} is not a necessary condition, as illustrated in Example \ref{E:2x2upper}.

Together with the current introduction, the paper consists of three sections. In Section 2 we recall some results on $*$-linear maps and minimal Hill representations from \cite{tHN2} that will be used throughout the paper. The proof of the main result, Theorem \ref{T:Main}, will be given in Section \ref{S:PtoCP} where we also prove that positivity and complete positivity coincide in a few additional cases and present various examples. Finally, in Section \ref{S:2x2} we illustrate our results by considering the case $n=q=2$.

We conclude this introduction with some words on notation and terminology and a few elementary formulas which can be found in most advanced linear algebra textbooks, cf., \cite{HJ85,HJ91,W16}. The standard $j$-th basis element in $\BF^n$ is denoted by $e_j^{(n)}$ or simply $e_j$ when the length is clear from the context. We write $\mathcal{E}_{ij}^{(n,m)}$ for the standard basis element of $\BF^{n \times m}$ with $1$ on position $(i,j)$ and zeros elsewhere, i.e., $\mathcal{E}_{ij}^{(n,m)}=e_i^{(n)}e_j^{(m)T}$, abbreviated to $\mathcal{E}_{ij}^{(n)}$ when $m=n$. With $\OneVec_n$ we indicate the all-1 vector of length $n$ and with $\BBone_{n \times m}$ the all-1 matrix of size $n \times m$, so that $\BBone_{n \times m}=\OneVec_{n}\OneVec_m^T$. Also here, we write $\BBone_n$ for $\BBone_{n \times n}$. Furthermore, $I_n$ denotes the $n \times n$ identity matrix and $P^{(n)}_{i,j}$ the permutation matrix of size $n \times n$ that interchanges the $i$-th and $j$-th row/ column, abbreviated to $P_{i,j}$ when there can be no confusion about the size.

For $A \in \BF^{n \times m}$ we write $A^T$ for its transpose, $A^*$ for its adjoint, $\overline{A}$ for its complex conjugate, $\kr{A}$ for its nullspace and $\im A$ for its range. We write $\cH_n$ for the $n \times n$ Hermitian matrices and $\cS_n$ for the $n \times n$ symmetric matrices. For $\BF=\BR$, of course, $\cH_n$ and $\cS_n$ coincide. With $A\geq 0$ (resp.\ $A>0$) we indicate that $A$ is positive semidefinite (resp.\ positive definite). Occasionally we will identify $\BF^n$ with $\BF^{n \times 1}$, so that matrix operations can be applied to vectors in $\BF^n$. The {\em Kronecker product} of matrices $A=\left[a_{ij}\right]\in\BF^{n \times m}$ and $B\in\BF^{k \times l}$ is defined as
\[
A\otimes B=\left[a_{ij}B\right]\in \BF^{(n k) \times (m l)},
\]
and the {\em Hadamard product} of matrices $A=\left[a_{ij}\right], B=\left[b_{ij}\right]\in\BF^{n \times m}$ is defined as
\[
A\circ B=\left[a_{ij}b_{ij}\right]\in\BF^{n \times m}.
\]
The \emph{vectorization} of a matrix $T \in \BF^{n \times m}$ is the vector $\vect_{n \times m }{(T)} \in \BF^{nm}$ defined as
\[
\vect_{n \times m}(T)= \sum_{j=1}^m\left(e_j^{(m)} \otimes I_n\right)Te_j^{(m)}.
\]
Note that the vectorization operator $\vect_{n \times m}$ defines an invertible linear map from $\BF^{n\times m}$ onto $\BF^{nm}$. If $m=n$ we just write $\vect_n$ and if the sizes are clear from the context, the indices are often left out. Moreover, we have
\begin{equation}\label{UnitVectId1}
\vect_{m \times n}\left(\cE^{(m,n)}_{l,k}\right)=
\vect_{m \times n}\left(e_l^{(m)}e_k^{(n)^T}\right)= e_k^{(n)} \otimes e_l^{(m)} =e_{(k-1)m +l}^{(nm)}.
\end{equation}
Finally, we define the {\em canonical shuffle} $\fC_n :\BF^{n^2} \to \BF^{n^2}$ which on pure tensors is defined as
\[
\fC_n (z \otimes x) = x \otimes z,
\]
and extended to $\BF^{n^2}$ by linearity.  Note that $\fC_n$ is a linear self-invertible map on $\BF^{n^2}$ which also satisfies $\fC_n^*=\fC_n$,  the matrix corresponding to $\fC$ is a signature matrix.

\section{Hill representations for $*$-linear matrix maps}\label{S:Hill}

Throughout this section $\cL$ is a linear matrix map of the form
\begin{equation}\label{cL}
\cL:\BF^{q\times q} \to \BF^{n\times n}.
\end{equation}
Recall that $\cL$ is called $*$-linear when $\cL(V^*)=\cL(V)^*$ for each $V\in \BF^{q\times q}$. In this section we review some results on $*$-linear matrix maps from \cite{tHN2} that will be used throughout the paper. In particular, we will discuss so-called minimal Hill representations that go back to the work of R.D. Hill in \cite{H69} and were further studied in \cite{H73,PH81,tHN2}. All results presented here can be found in \cite{tHN2}, or easily derived from results obtained there. We start with some general observations about $*$-linear matrix maps.

\subsection{$*$-Linear matrix maps}
Let $L$ and $\BL$ be the matricization and Choi matrix associated with the linear map $\cL$, as defined in \eqref{Matricization} and \eqref{Choi}, respectively. In that case $\cL$ can be expressed in terms of $L$ and $\BL$ via
\begin{equation}\label{Linv}
\cL(V)= \vect^{-1}\left(L\vect(V)\right),\quad V \in\BF^{q\times q}.
\end{equation}
and
\begin{equation}\label{BL to cL1}
\cL(V)=\sum_{i,j=1}^q  v_{ij}\BL_{ij},\quad V=[v_{ij}]\in\BF^{q\times q}.
\end{equation}
The latter can also be written as
\begin{equation}\label{BL to cL2}
\cL(V)=\left(\OneVec_q\otimes I_n\right)^T\left(\BL \circ \left(V \otimes \BBone_{n \times n}\right)\right)\left(\OneVec_q\otimes I_n\right).
\end{equation}
$*$-Linearity of $\cL$ can be expressed in terms of $L$ and $\BL$ as explained in the next result.

\begin{theorem}\label{T:*-linear}
Let $\cL$ be a linear map as in \eqref{cL}. Define $L\in\BF^{n^2 \times q^2}$ by \eqref{Matricization} and $\BL\in\BF^{nq \times nq}$ by \eqref{Choi}. Then the following are equivalent:
\begin{itemize}
  \item[(i)] $\cL$ is $*$-linear;
  \item[(ii)] $\BL\in\cH_{nq}$;
  \item[(iii)] $\ov{L}=\fC_n L \fC_q$.
\end{itemize}
Moreover, when $L$ is written as block matrix $L=\left[L_{ij}\right]$ with $L_{ij}=\left[\ell^{ij}_{kl}\right]\in\BF^{n \times q}$ then (iii) (and hence (i) and (ii)) is equivalent to
 \begin{equation}\label{HermChar2}
\ell_{kl}^{ij}=\ov{\ell}_{ij}^{kl},\quad 1 \le i,k \le n,\quad 1 \le j,l \le q.
\end{equation}
\end{theorem}

One way of interpreting \eqref{HermChar2} is that ``structural properties of $L$ as a block matrix reoccur at the level of the blocks.'' See Subsection 6.1 in \cite{tHN2} for more details and examples. We mention here two implications, using the notation of Theorem \ref{T:*-linear}:
\begin{itemize}
\item[(i)] We have $L_{ij}=0$ if and only if $\ell_{ij}^{kl}=0$ for all $k$ and $l$.

\item[(ii)] We have $L_{ij}=L_{rs}$ for $(i,j)\neq (r,s)$ if and only if $\ell^{kl}_{ij}=\ell^{kl}_{rs}$ for all $k$ and $l$.

\end{itemize}
Hence zero-structure in the block matrix structure of $L$ corresponds to the same zero-structure in all blocks $L_{ij}$, and similarly for repeating blocks. Thus, diagonal, upper triangular, lower triangular, Toeplitz, Hankel, circulant, etc.\ structure at the block matrix level of $L$ reoccurs in the blocks.

Recall that $P_{i,j}$ denotes the permutation matrix that interchanges the $i$-th and $j$-th row/column and note that $P_{i,j}=P_{i,j}^*=P_{i,j}^T=P_{i,j}^{-1}$. In the next lemma we describe how permutations in $\cL$ translate to $L$ and $\BL$. The identities follow directly from \eqref{Linv} and \eqref{BL to cL1}; see also Lemma 3.6 and Corollary 3.7 in \cite{tHN2}.

\begin{lemma}\label{L:Perm}
Let $\cL$ in \eqref{cL} be linear and define $L$ as in \eqref{Matricization} and $\BL$ as in \eqref{Choi}. For $1\leq i_1,i_2,k_1,k_2 \leq n$, $1\leq j_1,j_2,l_1,l_2 \leq q$ define the linear map $\wtil{\cL}:\BF^{q \times q} \to \BF^{n \times n}$ via $\wtil{\cL}(V)= P_{k_1,k_2}\cL(P_{l_1,l_2} V P_{j_1,j_2})P_{i_1,i_2}$ for all $V\in\BF^{q \times q}$. Then the matricization $\wtilL$ and Choi matrix $\wtil{\BL}$ associated with $\wtil{\cL}$ are given by
\[
\wtilL= \left(P_{i_1,i_2} \otimes P_{k_1,k_2}\right)L \left(P_{j_1,j_2} \otimes P_{l_1,l_2}\right)\ \mbox{ and }\
\wtil{\BL} = \left(P_{i_1,i_2} \otimes P_{j_1,j_2}\right)\BL \left(P_{k_1,k_2} \otimes P_{l_1,l_2}\right).
\]
In particular, when $i_r=k_r$, $j_r=l_r$ for $r=1,2$, it follows that $\cL$ is $*$-linear (respectively positive or completely positive) if and only if $\wtil{\cL}$ is $*$-linear (respectively positive or completely positive).
\end{lemma}

\subsection{Minimal Hill representations for $*$-linear maps}

In \cite{H69,H73} R.D. Hill studied linear matrix maps $\cL$ as in \eqref{cL} of the form
\begin{equation}\label{HillRep}
\cL(V)=\sum_{k,l=1}^m \BH_{kl}\, A_k V A_l^*,\quad V\in\BF^{q \times q},
\end{equation}
for matrices $A_{1},\dots,A_m\in\BF^{n \times q}$ and $\BH=[\BH_{kl}]_{k,l=1}^m\in\BF^{m\times m}$. We shall refer to \eqref{HillRep} as a {\em Hill representation} of $\cL$ and to $\BH$ as the associated Hill matrix. In case $m$ is such that there is no Hill representation of $\cL$ with a smaller number of matrices $A_k$, then we say that the Hill representation \eqref{HillRep} is minimal. It turns out that the minimal value of $m$ is $\rank \BL$, and in the sequel we shall restrict to minimal Hill representations, that is, we will restrict to the case that
\[
m:=\rank \,\BL.
\]
Note that in a minimal Hill representation \eqref{HillRep} the matrices $A_1,\ldots,A_m$ must be linearly independent. The next theorem collects a few of the main results from \cite{H73,PH81}, for $\BF=\BC$, see \cite{tHN2} for the general case.

\begin{theorem}\label{T:Hill}
The linear map $\cL$ in \eqref{cL} is $*$-linear if and only if $\cL$ admits a (minimal) Hill representation with $\BH$ Hermitian. Moreover, $\cL$ is completely positive if and only if $\cL$ admits a (minimal) Hill representation with $\BH$ positive definite.
\end{theorem}

The next result shows how the matrices $L$ and $\BL$ associated with $\cL$ can be expressed in terms of a minimal Hill representation of $\cL$.

\begin{proposition}\label{P:HillLBL}
Let $\cL$ be a $*$-linear map as in \eqref{cL} with a minimal Hill representation \eqref{HillRep}. Then the matricization $L$ and Choi matrix $\BL$ associated with $\cL$ are given by
\begin{equation}\label{LBL-Hill}
L=\sum_{k,l=1}^m \BH_{kl}\, \ov{A}_k\otimes A_l \ands \BL=\widehat{A}^*\BH^T\widehat{A},
\end{equation}
with $\widehat{A}^*:=\begin{bmatrix} \vect_{n \times q}\left(A_1\right) & \hdots & \vect_{n \times q}\left(A_m\right)  \end{bmatrix}\in \BF^{nq \times m}$. Moreover, $\whatA$ has full row rank and $\kr \whatA = \kr \BL$.
\end{proposition}

One of the main results in \cite{tHN2} is the following theorem which addressed the question of which matrices $A_1,\ldots,A_m$ can appear in a minimal Hill representation for $\cL$.  Decompose the matricization $L=\left[L_{ij}\right]$ of $\cL$ as in \eqref{Lblock}. The relation between $L$ and $\BL$, cf., Proposition 4.1 in \cite{tHN2}, implies that the columns of $\BL$ are vectorizations of the blocks $L_{ij}$ in $L$. Therefore, we have
\begin{equation}\label{mAlt}
m=\dim \tu{span}\{ L_{ij} \colon i=1,\ldots,n, \, j=1,...,q \} \subset \BF^{n\times q}.
\end{equation}

\begin{theorem}\label{T:HillA1Am}
Let $\cL$ be a $*$-linear map as in \eqref{cL} with a minimal Hill representation \eqref{HillRep}. Decompose the matricization $L=\left[L_{ij}\right]$ associated with $\cL$ as in \eqref{Lblock}. Then
\begin{equation}\label{AkCond}
\tu{span}\{ A_{k} \colon k=1,\ldots,m \}=\tu{span}\{ L_{ij} \colon i=1,\ldots,n, \, j=1,...,q \}.
\end{equation}
Moreover, for any choice of matrices $A_1,\ldots,A_m\in\BF^{n \times q}$ which satisfy \eqref{AkCond} there exists a matrix $\BH=[\BH_{kl}]\in\BF^{m \times m}$ so that $\cL$ is given by the corresponding minimal Hill representation \eqref{HillRep}.
\end{theorem}

The fact that all matrices $A_1,\ldots,A_m$ satisfying \eqref{AkCond} appear in a minimal Hill representation of $\cL$ can also be expressed in terms of the matrix $\whatA$, as in the following proposition; see \cite[Proposition 5.10]{tHN2}.

\begin{proposition}\label{P:whatAchar}
Assume $\cL$ as in \eqref{cL} is $*$-linear and let $m$ be the rank of the Choi matrix $\BL$ associated with $\cL$. Let $\whatA\in\BF^{m \times nq}$ with $\kr \whatA=\kr \BL$. Then $\whatA$ has full row rank, so that $\whatA\whatA^*$ is invertible, and we have $\BL=\whatA^* \BH^T \whatA$ with
\[
\BH^T=\left(\whatA\whatA^*\right)^{-1} \whatA \BL \whatA^* \left(\whatA\whatA^*\right)^{-1}.
\]
In particular, $\cL$ admits a minimal Hill representation \eqref{HillRep} with $A_k=\vect^{-1}_{n\times q} \left(\what{a}_k^T\right)$, $k=1,\ldots,m$, where $\what{a}_k$ is the $k$-th row of $\whatA$.
\end{proposition}

Next we will recall a construction of a minimal Hill representation of $\cL$ from \cite{tHN2} that will be important in the sequel. Select $L_1,\ldots,L_m\in\BF^{n \times q}$ so that
\begin{equation}\label{LkCond}
\tu{span}\{ L_{ij} \colon i=1,\ldots,n, \, j=1,...,q \} =\tu{span}\{ L_{k} \colon k=1,\ldots,m \}.
\end{equation}
Then there exits scalars $\al^{ij}_k, \be_{ij}^k\in\BF$ for $i=1,\ldots,n$, $j=1,...,q$ and $k=1,\ldots,m$, so that
\begin{equation}\label{LkLij}
L_k =\sum_{i=1}^n\sum_{j=1}^q \be_{ij}^k L_{ij},\quad L_{ij}=\sum_{k=1}^m \al^{ij}_k L_k.
\end{equation}
Now set
\begin{equation}\label{AkBk}
A_k=\left[\ov{\al}_{k}^{ij}\right]\in \BF^{n \times q} \quad \text{and} \quad B_k=\left[\be_{ij}^k\right]\in\BF^{n \times q} \quad \text{for} \quad k=1,\ldots,m.
\end{equation}
Then we have
\begin{equation}\label{LLk}
L=\sum_{k=1}^m \ov{A}_k \otimes L_k \ands L_k=\left(\OneVec_n \otimes I_n\right)^*\left(\left(B_k \otimes \BBone_{n \times q}\right) \circ L\right)\left(\OneVec_q\otimes I_q\right).
\end{equation}
Finally, we define the Hill matrix associated with the selection $L_1,\ldots,L_m$ as:
\begin{equation}\label{Hill}
\BH=\BH\left(\cL;L_1,\ldots,L_m\right):=\left[\OneVec_n^*\left(B_k \circ \ov{L}_l\right)\OneVec_q\right]_{k,l=1}^m\in\BF^{m \times m}.
\end{equation}

\begin{theorem}\label{T:HillConstruct}
Let $\cL$ be a $*$-linear map as in \eqref{cL} and select matrices $L_1,\ldots,L_m\in\BF^{n \times q}$ satisfying \eqref{LkCond}. Then $\cL$ is given by the minimal Hill representation \eqref{HillRep} with $A_1,\ldots,A_m$ as in \eqref{AkBk} and $\BH$ as in \eqref{Hill}. Moreover, all minimal Hill representations of $\cL$ are obtained in this way.
\end{theorem}

Finally, the minimal Hill representation of a $*$-linear map $\cL$ is unique up to an invertible $m \times m$ matrix, as explained in the next result; see Theorem 5.9 in \cite{tHN2} for an explicit formula for the invertible matrix $\Phi$.

\begin{theorem}\label{T:MinimalHillUnique}
Let the $*$-linear matrix map $\cL$ in \eqref{cL} be given by two minimal Hill representations \eqref{HillRep}, one with Hill matrix $\BH$ and associated matrices $A_1,\ldots, A_m$ and one with Hill matrix $\BH'$ and associated matrices $A'_1,\ldots, A'_m$. Then there exists an invertible matrix $\Phi\in\BF^{m\times m}$ so that
\[
(\Phi^* \otimes I_n)\mat{A_1\\\vdots\\ A_k}= \mat{A'_1\\\vdots\\ A'_k},\quad
\ov{\Phi}^*\whatA =\whatA',\quad \BH=\Phi \BH'\Phi^*,
\]
where $\whatA$ and $\whatA'$ are the matrices associated with $A_1,\ldots, A_m$ and $A'_1,\ldots, A'_m$ as in Proposition \ref{P:HillLBL}, respectively.
\end{theorem}

\section{Positive linear matrix maps that are completely positive}\label{S:PtoCP}

In this section we prove our main result, Theorem \ref{T:Main}, give several illustrative examples and prove a few more independent cases where positivity implies complete positivity. We start with some general observations, followed by proofs of the main result under additional constraints, leading eventually to a proof of the general case.

\subsection{General observations}

While it is easy to determine if a $*$-linear map is completely positive, via the Choi matrix, for positivity this is much less straightforward. We start this section with a necessary and sufficient criteria on the Choi matrix to determine whether $\cL$ is positive. The result is essentially contained in Propositions 3.1 and 3.6 of \cite{KMcCSZ19}; we add a proof for completeness.

\begin{proposition}\label{P:cL-Pos}
A $*$-linear map $\cL$ as in \eqref{cL} is positive if and only if the Choi matrix $\BL$ in \eqref{Choi} satisfies
\[
\left(z\otimes x\right)^* \BL \left(z\otimes x\right) \ge 0 \quad \mbox{for all}\quad x\in\BF^n \text{ and } z \in \BF^q,
\]
or equivalently,
\[
\left(I_q\otimes x\right)^* \BL \left(I_q\otimes x\right)\geq 0\quad \mbox{for all}\quad x\in\BF^n.
\]
\end{proposition}

\begin{proof}[\bf Proof]
For $V=\left[v_{ij}\right]_{i,j=1}^q \in \BF^{q\times q}$ and $x\in\BF^n$ we have
\begin{align*}
x^*\cL\left(V\right)x &= \sum_{i,j=1}^q x^*\BL_{ij}xv_{ij}  =\OneVec_q^*\left(\left(I_q \otimes x^*\right)\BL \left(I_q \otimes x\right) \circ V\right)\OneVec_q\\
&=\trace\left(\left(I_q \otimes x^*\right)\BL \left(I_q \otimes x\right) \ov{V}\right).
\end{align*}
The map $\cL$ is positive precisely when $x^*\cL\left(V\right)x\geq 0$ for all $x\in\BF^n$ and positive semidefinite $V\in \BF^{q \times q}$, hence if and only if $\trace(\left(I_q \otimes x^*\right)\BL \left(I_q \otimes x\right) V)\geq 0$ for all $x\in\BF^n$ and positive semidefinite $V\in \BF^{q \times q}$. It now follows by Fej\'{e}r's theorem, cf., \cite[Corollary 5.7.4]{HJ85}, that $\cL$ is positive if and only if $\left(I_q \otimes x^*\right)\BL \left(I_q \otimes x\right)\geq 0$ for all $x\in\BF^n$. Since $\cL$ is assumed to be $*$-linear, we have $\BL\in \cH_{nq}$ so that $\left(I_q \otimes x^*\right)\BL \left(I_q \otimes x\right)\in \cH_{q}$ for each $x\in\BF^n$. Therefore, this works for both $\BF=\BC$ and $\BF=\BR$.
\end{proof}

Using Proposition \ref{P:HillLBL} we see that
\[
\left(z\otimes x\right)^* \BL \left(z\otimes x\right)=\left(z\otimes x\right)^* \widehat{A}^*\BH^T \widehat{A} \left(z\otimes x\right)   \quad \mbox{for all}\quad x\in\BF^n \text{ and } z \in \BF^q.
\]
Thus, positivity of $\cL$ in \eqref{cL} is equivalent to $y^* \BH^T y\geq 0$ for all $y$ from the set
\begin{equation}\label{y-set}
\fY_{\whatA}:=
\{ y_{z,x}:=\widehat{A}(z \otimes x) \colon x \in \BF^{n},\, z \in \BF^q \},
\end{equation}
while complete positivity is equivalent to $\BH \geq 0$ (or equivalently $\BH^T \geq 0$). Whether a positive map $\cL$ is also completely positive thus depends on whether the set $\fY_{\whatA}$ contains enough vectors to conclude $\BH \geq 0$ from the fact that $y^* \BH^T y\geq 0$ for all $y\in \fY_{\whatA}$. Clearly this is the case when $\fY_{\whatA}=\BF^m$.

\begin{corollary}\label{C:fYwhatA=Fm}
Let $\cL$ be a positive $*$-linear map as in \eqref{cL}. Define $\whatA$ as in Proposition \ref{P:HillLBL} and $\fY_{\whatA}$ as in \eqref{y-set}. Then $\cL$ is completely positive in case $\fY_{\whatA}=\BF^m$, with $m$ the rank of the Choi matrix $\BL$.
\end{corollary}

The condition $\fY_{\whatA}=\BF^m$ is only sufficient, cf., Example \ref{E:2x2upper} below. However, by looking for classes of positive maps which always satisfy $\fY_{\whatA}=\BF^m$, and hence are completely positive, provides us with some classes where positivity automatically implies complete positivity. In particular, $\fY_{\whatA}=\BF^m$ holds for the class considered in our main result, Theorem \ref{T:Main}, for $\BF=\BC$, while for $\BF=\BR$ there is only a single case where this is not the case.

\begin{remark}\label{R:L1-LmchoiceIrel}
The matrix $\whatA$ depends on the choice of linearly independent matrices $L_1,\ldots,L_m$ satisfying \eqref{LkCond}. However, by Theorem \ref{T:MinimalHillUnique}, the ranges of two sets $\fY_{\whatA}$ in \eqref{y-set} corresponding to two choices of such matrices  $L_1,\ldots,L_m$ are connected via an invertible map in $\BF^{m \times m}$. Hence for the question whether $\fY_{\whatA}=\BF^m$, the particular choice of $L_1,\ldots,L_m$ is not relevant. Naturally, for the question whether $\cL$ is completely positive whenever $\cL$ is positive, the choice of $L_1,\ldots,L_m$ is also not relevant.
\end{remark}

\begin{remark}\label{R:BilinRange}
Note that $\im \whatA=\BF^m$, by Proposition \ref{P:HillLBL}, and hence $\tu{span}\, \fY_{\whatA}= \im \whatA=\BF^m$. However, the map $(z,x)\mapsto y_{z,x}=\widehat{A}(z \otimes x)$ is bilinear and it may well happen that its range is strictly contained in $\im \whatA$. Not much appears to be known about the ranges of bilinear maps. We have only been able to find \cite{R90}. Clearly, the range of such maps is closed under scalar multiplication, $\la y \in \fY_{\whatA}$ for all $y\in \fY_{\whatA}$ and $\la\in\BF$, but not necessarily under addition. As a result of this lack of clarity on the structure of ranges of bilinear maps, it is far from straightforward when $\fY_{\whatA}=\BF^m$ will occur.
\end{remark}

As a consequence of Proposition \ref{P:whatAchar}, the question whether $\fY_{\whatA}=\BF^m$ depends only on the kernel of the Choi matrix.

\begin{corollary}
Let $\cL$ and $\cL'$ be positive $*$-linear maps as in \eqref{cL} with Choi matrices $\BL$ and $\BL'$, respectively. Set $m=\rank \BL$, $m'= \rank \BL'$ and define $\fY_{\whatA}$ and $\fY_{\whatA'}$ as in \eqref{y-set}, with $\whatA$ and $\whatA'$ as in Proposition \ref{P:HillLBL}. Assume $\kr \BL \subset \kr \BL'$. If $\fY_{\whatA}=\BF^{m}$, then $\fY_{\whatA'}=\BF^{m'}$.
\end{corollary}

\begin{proof}[\bf Proof]
Note that $\kr \BL \subset \kr \BL'$ implies $m \geq m'$. Since for $\whatA\in\BF^{m \times nq}$ and $\whatA'\in\BF^{m' \times nq}$ we can take any matrices with $\kr \whatA=\kr\BL$ and $\kr \whatA'=\kr\BL'$, by Proposition \ref{P:whatAchar}, and thus, without loss of generality, we may assume the first $m'$ rows of $\whatA$ coincide with the rows of $\whatA'$. In that case
\[
y_{z,x}=\whatA (z \otimes x)=\mat{\whatA' (z \otimes x)\\ *}=:\mat{y_{z,x}'\\ *},\quad x \in \BF^{n},\, z \in \BF^q,
\]
with $*$ indicating an unspecified vector. It follows that if $\fY_{\whatA}=\BF^{m}$, i.e., any $y\in\BF^m$ is of the form $y_{z,x}$ for some $x \in \BF^{n},\, z \in \BF^q$, then also $\fY_{\whatA'}=\BF^{m'}$.
\end{proof}

In the remainder of this section we will only consider the case where the matrices $L_1,\ldots,L_m$  are selected among the entries $L_{ij}$ of the block matrix $L$ associated with $\cL$. Say $L_k=L_{i_k j_k}$ for $k=1,\ldots,m$ for a selection of distinct pairs $(i_k,j_k)$ with $i_k\in \{1,\ldots,n\}$ and $j_k\in \{1,\ldots,q\}$. With this selection of pairs we associate the matrix
\begin{equation}\label{Emat}
E:=
\begin{bmatrix}
e_{\left(j_1-1\right)n+i_1}^{(nq)^T}\\ \vdots \\ e_{\left(j_m-1\right)n+i_m}^{(nq)^T}
\end{bmatrix}\in \BF^{m \times nq}.
\end{equation}
Note that $E$ only depends on the positions of the matrices $L_1,\ldots,L_m$ in $L$ and not on the entries of the matrices $L_1,\ldots,L_m$. In the next lemma we specify what the matrix $\whatA$ looks like in this case.

\begin{lemma} \label{form of A_k}
Let $\cL$ as in \eqref{cL} be $*$-linear. Define $L$ as in \eqref{Matricization} and $\BL$ as in \eqref{Choi} and set $m=\rank \, \BL$. Let $L_1,\ldots,L_m$ be a selection of $m$ linearly independent matrices among the blocks $L_{ij}$ of $L$, say $L_k=L_{i_kj_k}$ for $i_k \in \{1,\ldots,n\}$ and $j_k \in \{1,...,q\}$. Then the matrices $A_1,\ldots,A_m$ in \eqref{AkBk}, with $\al^{ij}_k$ as in \eqref{LkLij}, are given by
\begin{equation*}
A_k = \mathcal{E}_{i_kj_k}^{(n,q)} + \sum_{\substack{(i,j) \neq (i_k,j_k)\\ k=1, \ldots,m}} \overline{\alpha}_k^{ij}\mathcal{E}_{ij}^{(n,q)}
\end{equation*}
and the matrix $\whatA$ defined in Proposition \ref{P:HillLBL} is given by
\begin{align*}
    \widehat{A} &= \begin{bmatrix}
        \vect_{n \times q}\left(\overline{A}_1\right)^T \\ \vdots \\ \vect_{n \times q}\left(\overline{A}_m\right)^T
    \end{bmatrix} = E+\sum_{\substack{(i,j) \neq (i_k,j_k)\\ k=1, \ldots,m}}\al^{ij}e_{(j-1)n+i}^{(nq)^T}\ \mbox{ with }\ \al^{ij}=\mat{\al^{ij}_1\\\vdots\\ \al^{ij}_m}\in\BF^m.
\end{align*}
\end{lemma}

\begin{proof}[\bf Proof]
Since $L_{i_k j_k}=L_k$ for each $k$, we have $\al^{i_k j_k}_k=1$ and $\al^{i_k j_k}_l=0$ for $l\neq k$. Including this in the formula $A_{k}=\sum_{i=1}^n \sum_{j=1}^q \ov{\al}_k^{ij} \cE_{ij}^{(n,q)}$ gives the formula for $A_k$.

Finally, making use of $\eqref{UnitVectId1}$ and the fact that the vectorization operator is linear which means it preserves linear combinations, yields the formula for $\whatA$.
\end{proof}

Note that in the formula for $\whatA$ the 'structural properties' related to the choice of $L_1,\ldots,L_m$ (contained in $E$) are separated from the additional data (in the form of the vectors $\al^{ij}$). Most of the results we obtain in this section relate to a pattern determined by the choice of the matrices $L_1,\ldots,L_m$, that is, on the matrix $E$ in \eqref{Emat}. However, often there are many ways to choose the linearly independent matrices $L_1,\ldots,L_m$. Hence the interpretation of these results is that they hold for linear maps $\cL$ so that a selection of $L_1,\ldots,L_m$ among the block matrices in $L$ can be made which satisfies the constraints.

\begin{remark}\label{R:Reorder}
In addition to how the matrices $L_1,\ldots,L_m$ are chosen among the block entries of $L$, also the numbering of the chosen matrices has an effect on $\whatA$, to be precise, reordering the selected matrices, corresponds to multiplying $\whatA$ on the left with a permutation matrix. Less straight forward is that we can also reorder the columns and rows of $L$ in any way we want. This will effect the block matrix entries of $L$, who will all have the same reordering of rows and columns as done at the level of the block rows and columns, but for the question considered in the present paper that is not relevant. That reordering of block rows and columns is indeed allowed, provided the rows and columns of the blocks are reordered in the same way, is a direct consequence of Lemma \ref{L:Perm}.
\end{remark}

\subsection{The case where $L_{ij}=0$ for all $(i,j)\neq (i_k,j_k)$.}\label{SubS:L_0=0}

In this subsection we consider the case where all non-zero blocks in $L$ are linearly independent. In other words, we prove Theorem \ref{T:Main} with $L_0=0$. We also include several examples. Now choose $L_1,\ldots,L_m$ to be the non-zero blocks in $L$, say $L_k=L_{i_k j_k}$ for $k=1,\ldots,m$, so that $L_{ij}=0$ for all $(i,j)\neq (i_k,j_k)$ for all $k$. In this case, by Lemma \ref{form of A_k}, $\whatA=E$ with $E$ the matrix in \eqref{Emat}. Note that
\begin{equation}\label{E(z,x)}
E(z\otimes x) =\mat{x_{i_1}z_{j_1}\\ \vdots\\ x_{i_m}z_{j_m}},\quad x\in\BF^n,\, z\in\BF^q.
\end{equation}

The next proposition characterizes when $\fY_{\whatA}=\BF^m$ and provides a proof of  Theorem \ref{T:Main} in the case considered in this subsection.

\begin{proposition}\label{P:ZeroPattern}
Let $\cL$ as in \eqref{cL} be $*$-linear. Define $L$ as in \eqref{Matricization} and $\BL$ as in \eqref{Choi} and set $m=\rank \, \BL$. Let $L_1,\ldots,L_m$ be a selection of $m$ linearly independent matrices among the blocks $L_{ij}$ of $L$, say $L_k=L_{i_kj_k}$ for $k=1,\ldots,m$. Assume $L_{ij}=0$ for all $i,j$ with $(i,j)\neq (i_k,j_k)$ for $k=1,\ldots,m$. Then $\fY_{\whatA}=\BF^m$ holds if and only if
the following condition holds:
\begin{equation}\label{C1}
\mbox{(C1)\ \  For each $k$ we have $j_l\neq j_k$ for each $l$ or $i_l\neq i_k$ for each $l$.}
\end{equation}
In particular, in case $\cL$ as in \eqref{cL} is a positive $*$-linear map such that its matricization $L$ has the above structure, then $\cL$ is also completely positive.
\end{proposition}

The result follows directly from the following lemma, since $\whatA=E$, which will also be of use in the sequel.

\begin{lemma}\label{L:fE1=Fm}
For the matrix $E$ in \eqref{Emat} associated with a selection $L_1,\ldots,L_m$ among the block entries of $L$ we have
\begin{equation}\label{fE1}
\fE_1:=\{E(z \otimes x): z \in \BF^q, \, x \in \BF^n \}=\BF^m
\end{equation}
if and only if condition (C1) in \eqref{C1} holds.
\end{lemma}

\begin{proof}[\bf Proof]
We first prove the necessity of (C1). Assume (C1) does not hold. We show that $\fE_1\neq \BF^m$. Since (C1) does not hold there exists a $k$ and $l_1,l_2\neq k$ so that $i_k=i_{l_1}$ and $j_k=j_{l_2}$ (note that $l_1=l_2$ can occur). Now take $y\in\BF^m$ with $y_k=0$ but $y_{l_1}$ and $y_{l_2}$ are non-zero. If $y$ were of the form $y=E (z \otimes x)$ for $z \in \BF^q$ and $x \in \BF^n$, then $x_{i_k}z_{j_k}=y_k=0$ implies either $x_{i_k}=0$ or $z_{j_k}=0$, but then also $y_{l_1}=x_{i_k}z_{j_{l_1}}=0$ or $y_{l_2}=x_{i_{l_2}}z_{j_k}=0$, in contradiction with our choice of $y$. Thus $\fE_1\neq \BF^m$ cannot occur if (C1) does not hold.

Now we turn to the sufficiency of (C1). Hence, assume (C1) holds. Take $y\in\BF^m$. For each $k$, when $j_l\neq j_k$ for each $l$ (case 1), take $z_{j_k}=y_k$ and $x_{i_k}=1$, while in case $i_l\neq i_k$ for each $l$ (case 2), take $z_{j_k}=1$ and $x_{i_k}=y_k$. Set all unspecified entries of $x\in\BF^n$ and $z\in\BF^q$ equal to zero. The requirements (case 1 or case 2) guarantee that the entries of $x$ and $z$ are only given one value. Indeed, assume for instance that $i_k=i_l$ for some $k\neq l$ (this is how $x_{i_k}$ could get two values). This means for both $k$ and $l$ that case 2 is not possible, hence case 1 applies, which yields $x_{i_k}=1=x_{i_l}$. A similar arguments holds when $j_k=j_l$.  Moreover, it is clear that $x_{i_k}z_{j_k}=y_k$ for each $k$. Hence $y=E(z \otimes x)$. It follows that $\fE_1=\BF^m$.
\end{proof}

Already for $n=q=2$ two cases are excluded by (C1) in the setting considered in this subsection. These are discussed in the next two examples. The first shows that $\fY_{\whatA}=\BF^m$ is not a necessary condition for a positive map to be completely positive in general.

\begin{example}\label{E:2x2upper}
Consider the case where $L$ is a $2 \times 2$ upper triangular block matrix
\[
L=\mat{L_1&L_2\\0&L_3} = \mat{a_1&b_1&a_2&b_2\\0&c_1&0&c_2\\0&0&a_3&b_3\\0&0&0&c_3} \in\BF^{4 \times 4}.
\]
The cases with $L$ a $2 \times 2$ block matrix with $m=3$ and the zero block in another position can be treated in a similar way, leading to the same conclusion. The zeros in the left-lower entries of $L_i$ follow because we want the associated map $\cL$ to be $*$-linear, but this requirement implies more. Indeed, we have
\[
\BL=\mat{a_1&0&a_2&a_3\\0&0&0&0\\b_1&0&b_2&b_3\\c_1&0&c_2&c_3},
\]
and $*$-linearity is equivalent to $\BL=\BL^*$, hence to $a_1,b_2,c_3\in\BR$ and $a_2=\ov{b}_1, a_3=\ov{c}_1, b_3=\ov{c}_2$. Thus, assuming $*$-linearity, we have
\[
L=\mat{a_1&b_1&\ov{b}_1&b_2\\0&c_1&0&c_2\\0&0&\ov{c}_1&\ov{c}_2\\0&0&0&c_3} \ands
\BL=\mat{a_1&0&\ov{b}_1&\ov{c}_1\\0&0&0&0\\b_1&0&b_2&\ov{c}_2\\c_1&0&c_2&c_3}
\]
with $a_1,b_2,c_3\in\BR$. In this case we have
\[
\whatA=E=\mat{1&0&0&0\\0&0&1&0\\0&0&0&1} \ands \whatA(z\otimes x)=\mat{x_1z_1\\x_1z_2\\x_2z_2},\quad x,z\in\BF^2.
\]
Confirming the result of Proposition \ref{P:ZeroPattern}, it follows that
\[
\fY_{\whatA}=\{y\in\BF^3 \colon y_2=0\ \Rightarrow\ y_1 y_2=0  \} \neq \BF^3.
\]
Despite the fact that $\fY_{\whatA}\neq \BF^3$, we claim that positivity of $\cL$ implies complete positivity, as we will prove now. Note that the positivity of $\cL$ corresponds to
\begin{equation}\label{PosConEx}
(I_2 \otimes x)^* \BL (I_2 \otimes x)=\mat{a_1|x_1|^2 & \ov{x}_1\left(x_1 \ov{b}_1 + x_2 \ov{c}_1\right) \\ x_1\left(\ov{x}_1 b_1 + \ov{x}_2 c_1\right) & x^* \sbm{b_2 & \ov{c}_2\\ c_2& c_3}x}\geq0,\quad x\in\BF^2,
\end{equation}
in particular, that $\sbm{b_2 & \ov{c}_2\\ c_2& c_3}\geq 0$. Assume $\cL$ is positive. We show $\cL$ is also completely positive. Clearly \eqref{PosConEx} implies $a_1\geq 0$.

Consider the case where $a_1=0$. Then \eqref{PosConEx} implies $\ov{x}_1\left(x_1 \ov{b}_1 + x_2 \ov{c}_1\right)=0$ for all $x_1,x_2\in\BF$, hence $b_1=c_1=0$. In that case $\BL\geq 0$ corresponds to $\sbm{b_2 & \ov{c}_2\\ c_2& c_3}\geq 0$, which we already established. Hence, for the remainder we may assume $a_1>0$.

Note that for $x$ with $x_1=0$ \eqref{PosConEx} only tells us that $c_3\geq0$, which we already know. Thus only if $x_1\neq 0$ can we get new information. In that case, without loss of generality we can take $x_1=1$ so that $x$ has the form $x=\sbm{1\\ \de}$ and \eqref{PosConEx} reduces to
\[
\mat{a_1 & \ov{b}_1 + \de \ov{c}_1 \\ b_1 + \ov{\de} c_1 & b_1+\de\ov{c}_2 + \ov{\de} c_2 + |\de|^2 c_3}\geq 0.
\]
Since $a_1>0$, this is the same as positivity of the Schur complement with respect to $a_1$, that is
\begin{align}
0&\leq a_1\left(b_1+\de\ov{c}_2 + \ov{\de} c_2 + |\de|^2 c_3\right)-\left(\ov{b}_1 + \de \ov{c}_1\right)\left( b_1 + \ov{\de} c_1\right) \notag\\
&= \left(a_1 b_1 -|b_1|^2\right)+\de\left(a_1 \ov{c}_2- \ov{c}_1 b_1\right) + \ov{\de}\left(a_1 c_2- c_1 \ov{b}_1\right) + |\de|^2\left(a_1c_3- |c_1|^2\right) \label{ineqEx}\\
&= \mat{1\\ \de}^* \mat{a_1 b_1 -|b_1|^2 & a_1 \ov{c}_2- \ov{c}_1 b_1 \\ \ov{a}_1 c_2- c_1 \ov{b}_1 &  a_1c_3- |c_1|^2} \mat{1\\ \de}.\notag
\end{align}
By scaling we obtain that
\[
x^* \mat{a_1 b_1 -|b_1|^2 & a_1 \ov{c}_2- \ov{c}_1 b_1 \\ \ov{a}_1 c_2- c_1 \ov{b}_1 &  a_1c_3- |c_1|^2} x \geq 0
\]
for all $x\in\BF^2$ with $x_1\neq 0$, while for $x_1=0$ (equivalently $a_1c_3- |c_1|^2 \geq 0$) this follows by taking $|\de|$ large enough in \eqref{ineqEx}. Hence, we conclude that
\[
\mat{a_1 b_1 -|b_1|^2 & a_1 \ov{c}_2- \ov{c}_1 b_1 \\ \ov{a}_1 c_2- c_1 \ov{b}_1 &  a_1c_3- |c_1|^2} \geq 0, \mbox{ so that } \mat{0&0&0\\ 0& a_1 b_1 -|b_1|^2 & a_1 \ov{c}_2- \ov{c}_1 b_1 \\ 0 & \ov{a}_1 c_2- c_1 \ov{b}_1 &  a_1c_3- |c_1|^2} \geq 0.
\]
The latter shows precisely that the Schur complement of $\BL$ with respect to $a_1$ is positive semidefinite, so that we can conclude that $\BL\geq 0$, or equivalently, that $\cL$ is completely positive, as claimed.
\end{example}

\begin{example}\label{E:2x2fullblock}
Take
\begin{equation*}
L=\begin{bmatrix} L_1 & L_2 \\ L_3 & L_4  \end{bmatrix}=\begin{bmatrix} \mathcal{E}_{11}^{(2)} & \mathcal{E}_{21}^{(2)} \\ \mathcal{E}_{12}^{(2)} & \mathcal{E}_{22}^{(2)}  \end{bmatrix}
=\mat{1&0&0&0\\0&0&1&0\\0&1&0&0\\0&0&0&1} \in \BF^{4 \times 4}.
\end{equation*}
Thus $n=q=2$ and $m=4$. In this case $\BL=\BH=L=L^{-1}.$ In particular, the Choi matrix $\BL$ is Hermitian, but not positive semidefinite, and hence the associated linear map $\cL$ is $*$-linear, but not completely positive. In fact, we have
\begin{equation*}
\cL(K)=\sum_{i,j=1}^2 k_{ij}\BL_{ij}=\begin{bmatrix} k_{11} & k_{21} \\ k_{12} & k_{22} \end{bmatrix}=K^T, \quad K=\begin{bmatrix} k_{11} & k_{12} \\ k_{21} & k_{22} \end{bmatrix} \in \BF^{2 \times 2}.
\end{equation*}
Hence, this choice of $L$ leads to the transpose map, which is one of the best known examples of a positive linear map which is not completely positive.

In this case we have $\whatA=E=L$, since $L_1,\ldots,L_4$ are linearly independent, and
\[
\fY_{\whatA}=\left\{\sbm{x_1z_1 \\ x_1z_2 \\ x_2z_1 \\ x_2z_2}: \, x_i,z_i \in \BF, \, i=1,2 \right\}.
\]
It is easily checked that $\mat{0&1&1&0}^T \not\in\fY_{\whatA}$ so that $\fY_{\whatA}\neq \BF^4$. Furthermore, for $x,z\in\BF^2$ we have
\[
(z\otimes x)^* \BL (z\otimes x) = \left(\ov{x}_1z_1+\ov{x}_2z_2\right)\left(\ov{\ov{x}_1z_1+\ov{x}_2z_2}\right) = |\ov{x}_1z_1+\ov{x}_2z_2|^2 \ge 0,
\]
verifying that $\cL$ is positive.
\end{example}

While there will be other choices of $L_1,\ldots,L_4$ where the $*$-linear map $\cL$ associated with this choice is both positive and completely positive, the example shows that the `structure' imposed by the positions of the linearly independent blocks does not guarantee it is always the case.

In fact, if this `pattern' is contained in a larger matrix $L$, then we can also construct an example where the same phenomenon occurs.

\begin{corollary}\label{C:2x2fullblock}
Let $(i_k,j_k)$, for $k=1,\ldots,m$, be a selection of distinct points in $\{1,\ldots,n\} \times \{1,\ldots,q\}$ so that there exist $k_1,\ldots,k_4$ with
\[
i_{k_1}=i_{k_2} \neq i_{k_3} = i_{k_4} \ands j_{k_1}=j_{k_3} \neq j_{k_2} = j_{k_4}.
\]
Then there exists a selection of matrices $L_1,\ldots,L_m\in\BF^{n\times q}$ so that the linear map $\cL$ with matricization $L=\left[L_{ij}\right]$ determined by $L_{i_k j_k}=L_k$, $k=1,\ldots,m$ and $L_{ij}=0$ for $(i,j)\neq(i_k,j_k)$ for all $k$ is a $*$-linear positive map that is not completely positive.
\end{corollary}

\begin{proof}[\bf Proof]
The result is proved by embedding Example \ref{E:2x2fullblock} into a larger matricization $L$. Possibly after relabeling $L_1,\ldots,L_m$ and reordering block rows and columns in $L$ (see Remark \ref{R:Reorder}) we can arrange to have $k_l=l$ for $l=1,\ldots,4$ and $i_1<i_3$, $j_1<j_2$. Since we will construct a case where $L_{ij}=0$ for $(i,j)\neq(i_k,j_k)$ for all $k$, we have that $\whatA=E$ so that $\whatA$ depends only on the choice of the pairs $(i_k,j_k)$. Moreover, $\whatA$ has the form
\begin{align*}
&\whatA=\mat{\whatA_0\\\whatA_1}=\mat{0&e^{(4)}_1&0&e^{(4)}_3&0&e^{(4)}_2&0&e^{(4)}_4&0\\ \whatA_{11}&0&\whatA_{12}&0&\whatA_{13}&0&\whatA_{14}&0&\whatA_{15}} \mbox{ with}\\
&\hspace*{2cm} \whatA_{11}\in\BF^{(m-4)\times ((j_1-1)n+ i_1-1) },\
\whatA_{15}\in\BF^{(m-4)\times ((q-j_2)n+n-i_3) },\\
&\whatA_{12}\in\BF^{(m-4)\times (i_3-i_1-1) },\,
\whatA_{13}\in\BF^{(m-4)\times ((j_2-j_1)n+i_1-i_3-1) },\,
\whatA_{14}\in\BF^{(m-4)\times (i_3-i_1-1) },
\end{align*}
and the $0$-s indicating zero matrices of appropriate size. For the Hill matrix take
\[
\BH=\mat{\BH_0&0\\0& \BH_1}\quad \mbox{with}\ \ \BH_0= \begin{bmatrix} \mathcal{E}_{11}^{(2)} & \mathcal{E}_{21}^{(2)} \\ \mathcal{E}_{12}^{(2)} & \mathcal{E}_{22}^{(2)}  \end{bmatrix}\in\BF^{4 \times 4},
\]
and $\BH_1\in\BF^{(m-4) \times (m-4)}$ any positive definite matrix. Then $\BL:=\whatA^* \BH^T \whatA$ has rank $m$, is selfadjoint but not positive semidefinite. Also, by construction, $L$ has the required structure, which follows from the relation between $\BL$ and $L$ explained in the introduction; see also \cite[Proposition 4.1 \& Section 3]{tHN2}. It follows that the linear map $\cL$ associated with $\BL$ and $L$ is $*$-linear but not completely positive. To see that $\cL$ is positive, note that
$\BL= \whatA_0^* \BH_0^T \whatA_0 + \whatA_1^* \BH_1^T \whatA_1$ with $\BL_1:=\whatA_1^* \BH_1^T \whatA_1$ positive semidefinite and $\BL_0:=\whatA_0^* \BH_0^T \whatA_0$ the matrix $\BL$ in Example \ref{E:2x2fullblock} with zero-rows and zero-columns added. Hence, if $(z \otimes x)^*\BL_0(z\otimes x)\geq 0$ for all $z ,x$, then $\cL$ is positive. Note that $\whatA_0(z \otimes x) = \mat{x_{i_1}z_{j_1} & x_{i_1}z_{j_2} & x_{i_3}z_{j_1} & x_{i_3}z_{j_2}}^T$, hence the argument used in Example \ref{E:2x2fullblock} also applies here, and it follows that $\cL$ is positive indeed.
\end{proof}

While Example \ref{E:2x2upper} shows that for $2 \times 2$ block upper triangular $L$ positivity of the associated $*$-linear map $\cL$ implies complete positivity, Corollary \ref{C:2x2fullblock} implies that for block upper triangular $L$ of larger size this is not the case.

\begin{example}
For $n=q>2$ it follows that there are matrices $L_{ij}\in\BF^{n \times n}$, for $i\leq j$, so that the linear map $\cL$ with  matricization $L$ given by the block upper triangular matrix
\[
L=\mat{L_{11}&\cdots&\cdots&L_{1n}\\ 0&\ddots&&\vdots\\ \vdots&\ddots&\ddots&\vdots\\ 0&\cdots&0&L_{nn}},
\]
where the non-zero blocks $L_{ij}$, $i\leq j$, are linearly independent, is positive but not completely positive. Indeed, for $n=q>2$ the right upper $2 \times 2$ corner forms a $2 \times 2$ block of linearly independent matrices.
\end{example}

\subsection{The case where $L_{ij}=L_0$ for all $(i,j)\neq (i_k,j_k)$:\ General setting.}

Now we start with the general case of Theorem \ref{T:Main}, i.e., where $L_0$ need not be the zero matrix. Set
\begin{equation}\label{fS}
\fS:=\tu{span}\{ L_{ij} \colon i=1,\ldots,n, \, j=1,...,q \}.
\end{equation}
Hence, we are assuming here that there exists a single matrix $L_0\in \fS$ so that $L_{ij}=L_0$ for all $(i,j)\neq (i_k,j_k)$ for $k=1,\ldots,m$. Although the case $L_0=0$ is covered in the previous subsection, we do allow this as a possibility here too. Since $\fS$ coincides with the span of $L_1,\ldots,L_m$, we can write $L_0=\sum_{k=1}^m \al_k L_k$ for $\al_1,\ldots,\al_m\in\BF$, which are uniquely determined by $L_0$, by the linear independence assumption. Note that in the case we consider here $\al_k^{ij}=\al_k$ for $(i,j)\neq (i_k,j_k)$ for all $k$ and it thus follows from Lemma \ref{form of A_k} that
\[
\whatA= E+ \al v^T
\]
with
\begin{align*}
E:=
\begin{bmatrix}
e_{\left(j_1-1\right)n+i_1}^{(nq)^T}\\ \vdots \\ e_{\left(j_m-1\right)n+i_m}^{(nq)^T}
\end{bmatrix}\in \BF^{m \times nq},\quad
\alpha:=\begin{bmatrix}\alpha_1 \\ \vdots \\ \alpha_m  \end{bmatrix} \in \BF^m,\\
\quad \text{and} \quad
v:= \OneVec_{nq}-\sum_{l=1}^m e_{\left(j_l-1\right)n+i_l}^{(nq)}= \OneVec_{nq}-E^T \OneVec_{m} \in \BF^{nq}.
\end{align*}
The latter formula for $v$ shows that $\whatA$ can also be written as
\begin{equation}\label{whatAforms1}
\whatA=E + \al \left(\OneVec_{nq}^T- \OneVec_{m}^T E\right) =\left(I_m-\al \OneVec_{m}^T\right) E + \al\OneVec_{nq}^T.
\end{equation}
In view of the second formula for $\whatA$ in \eqref{whatAforms1} we collect here some properties of the matrix $I_m-\al \OneVec_{m}^T$.

\begin{lemma}\label{L:Ipert}
The matrix $I_m-\al \OneVec_{m}^T$ is invertible if and only if $\sum_{k=1}^m \al_k\neq 1$, and in that case $\left(I_m-\al \OneVec_{m}^T\right)^{-1}=I_m +\left(1-\sum_{k=1}^m\alpha_k\right)^{-1} \alpha\OneVec_m^T$. In case $\sum_{k=1}^m \al_k= 1$ we have
\[
\kr I_m-\al \OneVec_{m}^T=\tu{span}\{\al\} \ands
\left(\im I_m-\al \OneVec_{m}^T\right)^\perp=\tu{span}\left\{\OneVec_{m}\right\}.
\]
\end{lemma}

\begin{proof}[\bf Proof]
The invertibility criteria and formula for the inverse follow directly from the Sherman–Morrison formula \cite{H89}. In case $\sum_{k=1}^m \al_k = 1$, since $I_m-\al \OneVec_{m}^T$ is a rank 1 perturbation of an invertible matrix, the kernel and cokernel are of dimension 1. It is easily verified that $\al$ is in $\kr I_m-\al \OneVec_{m}^T$ and $\OneVec_{m}\in \kr \left(I_m-\al \OneVec_{m}^T\right)^*=\kr I_m-\OneVec_{m} \al^*$, which prove the formulas for the kernel and cokernel of $I_m-\al \OneVec_{m}^T$.
\end{proof}

In the next two subsections we shall provide a proof for Theorem \ref{T:Main} in the case where $\sum_{k=1}^m \al_k\neq 1$ and the case where $\sum_{k=1}^m \al_k= 1$, respectively.

\subsection{The case where $L_{ij}=L_0$ for all $(i,j)\neq (i_k,j_k)$ and $\sum_{k=1}^m \al_k\neq 1$.}

In view of the formula for $\what{A}$ given in \eqref{whatAforms1}, when $I_m-\al \OneVec_{m}^T$ is invertible, it turns out useful to consider the subset of $\fE_1$ in \eqref{fE1} given by
\begin{equation}\label{fE2}
\fE_2:=\{E(z \otimes x): z \in \BF^q, \, x \in \BF^n,\, \OneVec_{nq}^T (z \otimes x)=0 \}.
\end{equation}
Note that
\begin{equation}\label{OneVecCon}
\OneVec_{nq}^T (z \otimes x)=\left(\sum_{j=1}^q z_j\right) \left(\sum_{i=1}^n x_i \right).
\end{equation}
Hence $\OneVec_{nq}^T (z \otimes x)=0$ holds if and only if $\sum_{j=1}^q z_j=0$ or $\sum_{i=1}^n x_i=0$. We now characterise when $\fE_2$ is equal to $\BF^m$.

\begin{lemma}\label{L:fE2=Fm}
We have $\fE_2=\BF^m$ if and only if (C1) in \eqref{C1} holds together with at least one of the following conditions:
\begin{itemize}
\item[(C2.1)] there exists a $i_0$ such that $ i_0 \neq i_k$ for all $k$;

\item[(C2.2)] there exists a $j_0$ such that $ j_0 \neq j_k$ for all $k$;

\item[(C2.3)] there exist distinct $k_1,k_2,k_3,k_4$ such that $i_{k_1}=i_{k_2} \neq i_{k_3} = i_{k_4}$;

\item[(C2.4)] there exist distinct $k_1,k_2,k_3,k_4$ such that $j_{k_1}=j_{k_2} \neq j_{k_3} = j_{k_4}$;

\item[(C2.5)]  we have $i_k\neq i_l$ for all $k\neq l$ and in case $q=1$ we must have $m\neq n$;

\item[(C2.6)] we have $j_k\neq j_l$ for all $k\neq l$ and in case $n=1$ we must have $m \neq q$;

\item[(C2.7)] there exists a $k$ such that $i_l \neq i_k$ and $j_l \neq j_k$ holds for all $l \neq k$ and $m>1$.

\end{itemize}
Furthermore, if (C1) holds and all seven conditions (C2.1)--(C2.7) do not hold, then, possibly rearranging the indices of $L_1,\ldots,L_m$, only one of the following three cases can occur:
\begin{itemize}
  \item[(i)] $q=1$ and $m=n$, which implies $L=\sbm{L_1\\ \vdots \\ L_n}$;

  \item[(ii)] $n=1$ and $m=q$, which implies $L=\sbm{L_1 & \cdots & L_q}$;

  \item[(iii)] $m=n+q-2$ with $n,q>2$ and there exist $1\leq r\leq n$, $1\leq s\leq q$ so that $L$ takes the form
  \[
L=\mat{
L_0&\cdots&L_0&L_q&L_0&\cdots&L_0\\
\vdots&&\vdots&\vdots&\vdots&&\vdots\\
L_0&\cdots&L_0&L_{q+r-2}&L_0&\cdots&L_0\\
L_1&\cdots&L_{s-1}&L_0&L_{s}&\cdots&L_{q-1}\\
L_0&\cdots&L_0&L_{q+r-1}&L_0&\cdots&L_0\\
\vdots&&\vdots&\vdots&\vdots&&\vdots\\
L_0&\cdots&L_0&L_{q+n-2}&L_0&\cdots&L_0
}.
\]

\end{itemize}
\end{lemma}

\begin{proof}[\bf Proof]
First we prove the sufficiency of (C1) together with one of the conditions (C2.1)--(C2.7). Hence assume (C1) holds. Then each $y\in\BF^m$ is of the form $y=E(z \otimes x)$ for some $z \in \BF^q$ and $x \in \BF^n$, since $\fE_1=\BF^m$, by Lemma \ref{L:fE1=Fm}. Now, under each of the additional assumptions (C2.1)--(C2.7) we have to show that we can in fact obtain $y=E(z \otimes x)$ with the additional constraint $\OneVec_{nq}^T (z \otimes x)=0$, or, equivalently, $\sum_{j=1}^q z_j=0$ or $\sum_{i=1}^n x_i=0$.

Assume (C2.1). Let $y\in\BF^m$ be given as $y=E(z \otimes x)$ for some $z \in \BF^q$ and $x \in \BF^n$. Let $i_0$ be such that $i_0 \neq i_k$ for all $k$. From \eqref{E(z,x)} it follows that the validity of $y=E(z \otimes x)$ holds independent of the value of $x_{i_0}$, so that $x_{i_0}$ can be modified to have $\sum_{i=1}^n x_i=0$.

Assume (C2.2). This follows by an argument similar to that for Condition (C2.1).

Assume (C2.4). Then there exist $r,s>1$ (but $r+s\leq n$) so that, possibly after rearranging  indices of $L_1,\ldots,L_m$ as well as block rows and columns in $L$ (see Remark \ref{R:Reorder}), $j_k=1$ and $i_k=k$ for $k=1,\ldots,r$ and $j_{r+l}=2$ and $i_{r+l}=r+l$ for $l=1,\ldots,s$, while $j_p\neq 1,2$ whenever $p>r+s$. Since $(C1)$ holds, also $i_p\neq 1,\ldots, r+s$ for $p>r+s$. Let $y\in\BF^m$ be given as $y=E(z \otimes x)$ for some $z \in \BF^q$ and $x \in \BF^n$. Then $y_k= x_k z_1$ for $k=1,\ldots,r$ and $y_{r+l}= x_{r+l} z_2$ for $l=1,\ldots,s$ while $z_1,z_2,x_1,\ldots,x_{r+s}$ do not occur in the factorisations $y_p=x_{i_p}z_{j_p}$ for $p>r+s$. We can now adjust $z_1$ and $z_2$ so that $z_1,z_2\neq 0$ and $\sum_{j=1}^q z_j =0$ and then redefine $x_k:=y_k/z_1$ for $k=1,\ldots,r$ and $x_{r+l}:=y_{r+l}/z_2$ for $l=1,\ldots,s$ resulting in new vectors $x$ and $z$ so that still $y=E(z \otimes x)$ while also $\sum_{j=1}^q z_j =0$.

Assume (C2.3). This follows by an argument similar to that for Condition (C2.4).

Assume (C2.5). In this case, since $i_k\neq i_l$ for all $k\neq l$ we can chose $x_{i_1},\ldots,x_{i_m}$ independent of one another. Also, we have $m\leq n$. In case $q=1$ and $m\neq n$ we have $m<n$ and thus we are in case (C2.1). Assume $q>1$. Let $y\in\BF^m$. Now choose $z_1=1-q\neq 0$ and $z_2=\cdots=z_q=1$, so that $\sum_{j=1}^q z_j=0$ and all $z_j$ are non-zero, and take $x\in\BF^n$ with $x_{i_k}=y_{k}/z_{j_k}$ for $k=1,\ldots,m$ with the other entries chosen arbitrarily. In this case we have $y=E(z \otimes x)$ while also $\OneVec_{nq}^T (z \otimes x)=0$ because $\sum_{j=1}^q z_j=0$.

Assume (C2.6). This follows by an argument similar to that for Condition (C2.5).

Assume (C2.7). Since all other cases are covered, we may assume these conditions are not satisfied, in particular, assume (C2.5) and (C2.6) do not hold. By excluding the case $m=1$, reasoning as for (C2.4), again after possibly rearranging indices of $L_1,\ldots,L_m$ as well as block columns and rows in $L$ (see Remark \ref{R:Reorder}), it follows that there exist $r,s>1$ (with $r+s+1\leq m$) so that $i_k=k$ and $j_k=1$ for $k=1,\ldots,r$, $i_{r+l}=r+1$ and $j_{r+l}=l+1$ for $l=1,\ldots,s$ and $i_{r+s+1}=r+2$, $j_{r+s+1}=s+2$, while $i_p\neq 1,\ldots, r+2$ and $j_p\neq 1,\ldots, s+2$ for $p> r+s+1$. Let $y\in\BF^m$ be given as $y=E(z \otimes x)$ for some $z \in \BF^q$ and $x \in \BF^n$. Then $y_k= x_k z_1$ for $k=1,\ldots,r$, $y_{r+l}= x_{r+1} z_{1+l}$ for $l=1,\ldots,s$ and $y_{r+s+1}=x_{r+2}z_{s+2}$, while $z_1,\ldots,z_{s+2},x_1,\ldots,x_{r+2}$ do not occur in the factorisations $y_p=x_{i_p}z_{j_p}$ for $p>r+s+1$. We now adjust $x$ and $z$ depending on whether $\sum_{j\neq s+2} z_j$ is zero or not. In case $\sum_{j\neq s+2} z_j\neq 0$, we just need to change $z_{s+2}$ and $x_{r+2}$ to  $z_{s+2}:=-\sum_{j\neq s+2} z_j$ and $x_{r+2}:=-y_{r+s+1}/\left(\sum_{j\neq s+2} z_j\right)$. In case $\sum_{j\neq s+2} z_j= 0$, select a $a\neq 0, z_1$, and redefine $z_{s+2}:=a$, $x_{r+2}=y_{r+s+1}/a$, $x_k:= y_k/(z_1-a)$ for $k=1,\ldots,r$ and change $z_1$ to $z_1-a$. In both cases it is easy to verify that the adjusted vectors $x$ and $z$ still satisfy $y=E(z \otimes x)$ while now also $\sum_{j=1}^q z_j=0$.

Now we prove the necessity claim. Since $\fE_2\subset\fE_1$, it is clear from Lemma \ref{L:fE1=Fm} that (C1) is a necessary condition for $\fE_2=\BF^m$ to hold. Hence, assume (C1) holds. To prove necessity we show that in case the conditions (C2.1)--(C2.7) all do not hold, then we do not have $\fE_2=\BF^m$. Hence assume the conditions (C2.1)--(C2.7) do not hold.

As a first step we show that this corresponds to the cases (i)--(iii). It is easily verified that the cases (i)--(iii) satisfy (C1) whilst not satisfying any of (C2.1)--(C2.7). Conversely, assume (C1) holds and non of (C2.1)-(C2.7) hold. This is equivalent to (C1) holding and the negation of each of (C2.1)-(C2.7) being true, which are given by:
\begin{itemize}
\item[$\neg$(C2.1)] for all $i$ there exists a $k$ such that $ i = i_k$;

\item[$\neg$(C2.2)] for all $j$ there exists a $k$ such that $j= j_k$;

\item[$\neg$(C2.3)] if there exists distinct $k_1,k_2,k_3,k_4$ such that $i_{k_1}=i_{k_2}$ and $i_{k_3}=i_{k_4}$, then $i_{k_1}=i_{k_2}=i_{k_3}=i_{k_4};$

\item[$\neg$(C2.4)] if there exists distinct $k_1,k_2,k_3,k_4$ such that $j_{k_1}=j_{k_2}$ and $j_{k_3}=j_{k_4}$, then $j_{k_1}=j_{k_2}=j_{k_3}=j_{k_4};$

\item[$\neg$(C2.5)] there exist $k\neq l$ such that $i_k = i_l$ or $q=1$ and $m = n$;

\item[$\neg$(C2.6)] there exist $k\neq l$ such that $j_k = j_l$ or $n=1$ and $m = q$;

\item[$\neg$(C2.7)] for all $k$ there exists a $l \neq k$ such that $i_k = i_l$ or $j_k = j_l$.

\end{itemize}

First assume we have $q=1$. By $\neg$(C2.1) it follows that necessarily $m=n$, so that we are in case (i). Likewise, if $n=1$, then $L$  must be as in (ii).

In case $q=2$, because of $\neg$(C2.2), either (C2.3) or (C2.5) holds. Hence $q=2$ cannot occur. Similarly, $n=2$ cannot occur.

In the case where $n> 2$ and $q > 2$, following $\neg$(C2.1) and $\neg$(C2.2) we know every row and column of $L$ contain at least one $L_k$. Furthermore, following $\neg$(C2.5) and $\neg$(C2.6) we know there is at least one row and one column of $L$ that contain more than one $L_k$ and together with $\neg$(C2.3) and $\neg$(C2.4) we know there can be at most one such row and column. Lastly, considering $\neg$(C2.7) we have for every row or column of $L$ more than one $L_k.$ Combining these requirements above, together with (C1) allows for no other form of $L$ than that in (iii).

To complete the proof we show that in each of the cases (i)--(iii) we do not have $\fE_2 \neq \BF^m$. In case (i), we have \begin{equation*}
    y = E\left(z \otimes x\right) = \begin{bmatrix}
        x_1z_1 \\ \vdots \\x_nz_1
    \end{bmatrix}.
\end{equation*}Thus we need $x_i \in \BF$ to be free and $z_1 \neq 0$ to get $\fE_2=\BF^m$, proving that we cannot ensure $\fE_2 = \BF^m$ and $\sum_{i=1}^nx_i=0$ or $z_1=0.$ In a similar manner it can be proven that case (ii) implies $\fE_2 \neq \BF^m,$ by swapping the roles of $z_j$ and $x_i$.

Finally, assume $L$ has the form as in case (iii). For $x\in\BF^n$ and $z\in\BF^q$ we have
\begin{multline*}
E(z \otimes x) = \left[\begin{matrix}
        x_{r}z_1 & \hdots & x_{r}z_{s-1} & x_{r}z_{s+1} & \hdots & x_{r}z_{q} \end{matrix} \right. \\ \left. \begin{matrix}  x_1z_{s} & \hdots & x_{r-1}z_s & x_{r+1}z_{s} & \hdots & x_nz_s
    \end{matrix}\right]^T.
\end{multline*}
Since $q,n>2$ we can choose $y\in \BF^m$ such that all entries of $y$ are non-zero while $\sum_{j=1}^{q-1}y_j=0$ and $\sum_{j=q}^{q+n-2}y_j=0$. Let $x\in\BF^n$ and $z\in\BF^q$ such that $y=E (z\otimes x)$. Then all entries of $x$ and $z$ must be non-zero. Also, $0=\sum_{j=1}^{q-1}y_j=x_r\sum_{j\neq s}z_j$. Since $x_r\neq 0$, we have $\sum_{j\neq s}z_j=0$. Therefore, $\sum_{j=1}^q z_j=z_{s}\neq 0$. Similarly, from $\sum_{j=q}^{q+n-2}y_j=0$ one obtains that $\sum_{i=1}^n x_i=x_{r}\neq 0$. Hence $y\neq \fE_2$.
\end{proof}

It is now easy to prove Theorem \ref{T:Main} when $\sum_{k=1}^m \al_k \neq 1$ under the additional condition that one of (C2.1)--(C2.7) given in Lemma \ref{L:fE2=Fm} holds.

\begin{lemma}\label{L:SumAl_k=1}
Assume (C1) in \eqref{C1} holds together with one of (C2.1)--(C2.7) in Lemma \ref{L:fE2=Fm}. For the matrix $L_0\in \fS$ let $\al_1,\ldots,\al_m\in\BF$ so that $L_0=\sum_{k=1}^m \al_k L_k$. Assume in addition that $\sum_{k=1}^m \al_k \neq 1$. Then $\fY_{\whatA}=\BF^m$.
\end{lemma}

\begin{proof}[\bf Proof]
Let $y\in\BF^m$. By Lemma \ref{L:Ipert} the matrix $I_m-\al \OneVec_{m}^T$ is invertible. Hence $y=\left(I_m-\al \OneVec_{m}^T\right)\what{y}$ for some $\what{y}\in\BF^m$. By Lemma \ref{L:fE2=Fm} we have $\fE_2=\BF^m$. Thus $\what{y}=E(z \otimes x)$ for some $z\in\BF^q$ and $x\in\BF^n$ with $\OneVec_{nq}^T(z\otimes x)=0$. This implies that
\[
\whatA (z\otimes x) = \left(I_m-\al \OneVec_{m}^T E\right) (z \otimes x) + \al \OneVec_{nq}^T(z\otimes x) = \left(I_m-\al \OneVec_{m}^T E\right) \what{y} +0 =y.
\]
Hence $y\in\fY_{\whatA}$. Thus $\fY_{\whatA}=\BF^m$.
\end{proof}

Hence, by the last statement of Lemma \ref{L:fE2=Fm}, it remains to prove that $\fY_{\whatA}=\BF^m$ in the setting of the present subsection for the cases (i), (ii) and (iii) of Lemma \ref{L:fE2=Fm}. For (i) and (ii) this is already covered by Proposition \ref{P:ZeroPattern}, also when $\sum_{k=1}^m \al_k=1$.

\begin{lemma}\label{L:Case (i) and (ii)}
Assume $L$ is as in (i) or (ii) in Lemma \ref{L:fE2=Fm}. Then $\fY_{\whatA}=\BF^m$.
\end{lemma}

\begin{proof}[\bf Proof]
For both (i) and (ii) the block entries of $L$ are assumed to be linearly independent and (C1) is satisfied, hence the claim follows from Proposition \ref{P:ZeroPattern}.
\end{proof}

Finally, we prove the claim for $L$ as in (iii) of Lemma \ref{L:fE2=Fm}. First we prove a lemma that will also be of use when we consider the case where $\sum_{k=1}^m \al_k =1$.

\begin{lemma}\label{L:Case (iii) prep}
Assume $L$ is as in case (iii) of Lemma \ref{L:fE2=Fm} and define $E$ as in \eqref{Emat}. Then for each $y^\circ\in\BF^m$ and $x_r,z_s\neq 0$ there exist unique $x\in \BF^n$ and $z\in\BF^q$ so that $y^\circ=E(z\otimes x)$ and so that $x_r$ and $z_s$ correspond to the $r$-th and $s$-th components of $x$ and $z$, respectively. In that case, we have
\begin{equation}\label{factForm}
\OneVec_{nq}^T (z \otimes x) =
\left(\frac{\wtil{y}^\circ_2}{z_s} + x_r \right)
\left(\frac{\wtil{y}^\circ_1}{x_r} + z_s\right),\quad
\mbox{where}\quad  \wtil{y}^\circ_1=\sum_{k=1}^{q-1} y^\circ_k,\quad  \wtil{y}^\circ_2=\sum_{k=q}^{m} y^\circ_k.
\end{equation}
\end{lemma}

\begin{proof}[\bf Proof]
The equation $y^\circ=E(z\otimes x)$ with $x_r,z_s\neq 0$ implies that $y^\circ_k=x_rz_k$, so that $z_k=y^\circ_k/x_r$, for $k=1, \ldots, s-1;$ $y^\circ_k=x_rz_{k+1}$, so that $z_{k+1}=y^\circ_k/x_r$, for $k=s, \ldots, q-1$; $y^\circ_{q-1+k}=x_{k}z_{s}$, so that $x_k=y^\circ_{q-1+k}/z_s$, for $k=1, \ldots, r-1$ and $y^\circ_{q-1+k}=x_{k+1}z_{s}$, so that $x_{k+1}=y^\circ_{q-1+k}/z_s$, for $k=r, \ldots, n-1$. Hence, $x$ and $z$ exist and are uniquely determined by $y^\circ$ and  $x_r,z_s$, as claimed. Moreover, we have
\[
\sum_{i=1}^n x_i = \sum_{i=1}^{r-1} x_i + x_r + \sum_{i=r+1}^n x_i = \sum_{k=1}^{n-1}\frac{y^\circ_{q-1+k}}{z_s} + x_r = \frac{\wtil{y}^\circ_2}{z_s} + x_r,
\]
and a similar computation shows that $\sum_{j=1}^q z_j=\frac{\wtil{y}^\circ_1}{x_r} + z_s$. Then \eqref{factForm} follows directly from \eqref{OneVecCon}.
\end{proof}

\begin{lemma}\label{L:Case (iii) inv}
Assume $L$ is as in case (iii) of Lemma \ref{L:fE2=Fm} and $\sum_{k=1}^m \alpha_k \neq1$. Then for $\BF=\BC$ we have  $\fY_{\whatA}=\BC^m$, while for $\BF=\BR$ we have  $\fY_{\whatA}=\BR^m$ if and only if
\begin{equation}\label{al12}
4\wtil{\al}_1 \wtil{\al}_2\leq 1,\quad \mbox{with}\quad \wtil{\al}_1=\sum_{k=1}^{q-1} \al_k,\ \wtil{\al}_2=\sum_{k=q}^{m} \al_k.
\end{equation}
Moreover, in case $4\wtil{\al}_1 \wtil{\al}_2> 1$, we have
\begin{equation*}\label{fY case (iii)}
\fY_{\whatA}=\left\{y\in\BR^m \colon
\sum_{k=1}^{q-1} y_k\neq 0 \mbox{ or }\sum_{k=q}^{m} y_k\neq 0,\mbox{ or }\sum_{k=1}^{q-1} |y_k|=0,\mbox{ or }\sum_{k=q}^{m} |y_k|=0
\right\}.
\end{equation*}
\end{lemma}

\begin{proof}[\bf Proof]
Assume $L$ has the form as in case (iii) in Lemma \ref{L:fE2=Fm} and $\sum_{k=1}^m \alpha_k\neq1$.
Fix a $y\in\BF^m$. In this case the matrix $I_m - \alpha \OneVec_m^T$ is invertible, by Lemma \ref{L:Ipert}, and hence for each $\la\in\BF$ we can find a unique $y_\la^\circ=y^\circ\in\BF^m$ so that
\[
y -\al \la = \left(I_m - \alpha \OneVec_m^T\right)y^\circ,\quad \mbox{that is} \quad y= \left(I_m - \alpha \OneVec_m^T\right)y^\circ  +\al \la.
\]
Now write $y^\circ=E (z\otimes x)$ with $x_r,z_s\neq 0$ so that $x$ and $z$ are uniquely determined by $y^\circ$ and $x_r,z_s$, that is, by $y,\la,x_r,z_s$, and formula \eqref{factForm} applies. Define $\wtil{\al}_1$ and $\wtil{\al}_2$ as in \eqref{al12} and set
\begin{align}\label{tilaly}
\wtil{\al}=\sum_{k=1}^m \al_k,\quad
\wtil{y}=\sum_{k=1}^m y_k,\quad \wtil{y}_1=\sum_{k=1}^{q-1} y_k,\quad \wtil{y}_2=\sum_{k=q}^{m} y_k.
\end{align}
Using the inversion formula from Lemma \ref{L:Ipert} we obtain that
\begin{align*}
y^\circ&=\left(I_m-\alpha \OneVec_m^T\right)^{-1}\left(y -\la \alpha\right)
= \left(I_m +(1-\wtil{\al})^{-1} \alpha\OneVec_m^T\right)\left(y -\la \alpha\right) \\
&= \left(y -\la \alpha\right) + \frac{\wtil{y} -\la \wtil{\al}}{1-\wtil{\al}}\al
=y + \frac{-\la(1-\wtil{\al}) + \wtil{y} -\la \wtil{\al}}{1-\wtil{\al}}\al
=y + \frac{\wtil{y} -\la}{1-\wtil{\al}}\al.
\end{align*}
Therefore, we have that
\[
\wtil{y}_1^\circ=\sum_{k=1}^{q-1} y_k^\circ = \wtil{y}_1 + \frac{\wtil{y} -\la}{1-\wtil{\al}}\wtil{\al}_1 \ands
\wtil{y}_2^\circ=\sum_{k=q}^{m} y_k^\circ = \wtil{y}_2 + \frac{\wtil{y} -\la}{1-\wtil{\al}}\wtil{\al}_2.
\]
We now show how the variables $\la,x_r,z_s$, subject to $x_r,z_s\neq 0$, can be chosen in such a way that $\OneVec_{nq}^T (z \otimes x)=\la$, that is, using \eqref{factForm}, that
\begin{equation}\label{laId2}
\left(\frac{\wtil{y}_2 + \frac{\wtil{y} -\la}{1-\wtil{\al}}\wtil{\al}_2}{z_s} + x_r \right)
\left(\frac{\wtil{y}_1 + \frac{\wtil{y} -\la}{1-\wtil{\al}}\wtil{\al}_1}{x_r} + z_s\right)=
\la.
\end{equation}
Set $\rho:=\frac{\wtil{y}-\la}{1-\wtil{\alpha}},$ so that $\la = \rho\left(\wtil{\alpha}-1\right)+ \wtil{y}$. Note that varying $\la$ in $\BF$, $\rho$ obtains all values in $\BF$ as well. Then \eqref{laId2} translates to
\begin{align*}
\rho\left(\wtil{\alpha}-1\right)+ \wtil{y}
&=\left( \frac{\wtil{y}_2 + \rho\wtil{\al}_2}{z_s} + x_r \right)
\left(\frac{\wtil{y}_1 + \rho\wtil{\al}_1}{x_r} + z_s \right)\\
&=\frac{\left(\wtil{y}_2 + \rho\wtil{\al}_2\right)\left(\wtil{y}_1 + \rho\wtil{\al}_1\right)}{z_s x_r} + \wtil{y}_2 + \rho\wtil{\al}_2 + \wtil{y}_1 + \rho\wtil{\al}_1 + z_s x_r\\
&=\frac{\left(\wtil{y}_2 + \rho\wtil{\al}_2\right)\left(\wtil{y}_1 + \rho\wtil{\al}_1\right)}{z_s x_r} + \rho\wtil{\al} + \wtil{y} + z_s x_r.
\end{align*}
Multiplying with $z_s x_r$ on both sides and rearranging terms this yields
\begin{equation}\label{rhoId1}
(z_s x_r)^2 + \rho z_s x_r + \left(\wtil{y}_2 + \rho\wtil{\al}_2\right)\left(\wtil{y}_1 + \rho\wtil{\al}_1\right)=0.
\end{equation}
For $\BF=\BC$ it is easy to find a $\rho$ so that the above equation in $z_s x_r$ has a non-zero solution, after which it remains to factor this solution to obtain non-zero $z_s$ and $x_r$ so that \eqref{laId2} holds.

For $\BF=\BR$, to obtain a non-zero real solution we need to find  $\rho$ so that in addition
\[
\left(1-4\wtil{\al}_1 \wtil{\al}_2\right)\rho^2-4\left(\wtil{y}_1 \wtil{\al}_2 + \wtil{y}_2 \wtil{\al}_1\right)\rho -4 \wtil{y}_1 \wtil{y}_2=\rho^2- 4 \left(\wtil{y}_2 + \rho\wtil{\al}_2\right)\left(\wtil{y}_1 + \rho\wtil{\al}_1\right)>0,
\]
or the left-hand side of the inequality equal to 0 for a $\rho\neq 0$.  Thus, the question is whether
\begin{equation}\label{f}
f(\rho)=(1-4\wtil{\al}_1 \wtil{\al}_2)\rho^2-4(\wtil{y}_1 \wtil{\al}_2 + \wtil{y}_2 \wtil{\al}_1)\rho -4 \wtil{y}_1 \wtil{y}_2
\end{equation}
attains positive values or is equal to zero at a non-zero point $\rho$. We consider three cases.\smallskip

\paragraph{\bf Case 1} Assume $1-4\wtil{\al}_1 \wtil{\al}_2>0$. In this case $f$ is a polynomial of degree two with positive main coefficient, hence positive values are attained for $|\rho|$ large enough.\smallskip

\paragraph{\bf Case 2} Assume $1-4\wtil{\al}_1 \wtil{\al}_2=0$. Then $\wtil{\al}_1$ and $\wtil{\al}_2$ are non-zero and of the same sign. If $\wtil{y}_1 \wtil{\al}_2 + \wtil{y}_2 \wtil{\al}_1\neq 0$, then $f$ is a polynomial of degree one and thus positive values are attained. In case $\wtil{y}_1 \wtil{\al}_2 + \wtil{y}_2 \wtil{\al}_1= 0$, $f$ is a constant function and two situations can occur: (1) $\wtil{y}_1 =\wtil{y}_2=0$ in which case $f$ is the zero-function so that for $\rho$ we can take any non-zero value, and (2) $\wtil{y}_1,\wtil{y}_2\neq 0$ with $\wtil{y}_1$ and $\wtil{y}_2$ of opposite sign (since $\wtil{\al}_1$ and $\wtil{\al}_2$ have the same sign) in which case the value of $f$ is $-4 \wtil{y}_1 \wtil{y}_2>0$ so that $\rho$ can be taken arbitrarily.\smallskip

\paragraph{\bf Case 3} Finally, assume $1-4\wtil{\al}_1 \wtil{\al}_2<0$. In this case $f$ is a polynomial of degree two with a negative main coefficient. Also $\wtil{\al}_1$ and $\wtil{\al}_2$ are non-zero and of the same sign. We need to show that in this case $\fY_{\whatA}$ is as stated.

First consider $y\in\BR^m$ with $\wtil{y}_1$ and $\wtil{y}_2$ not both zero. The maximum of $f$ is attained at $\rho_0=
2\left(\wtil{y}_1 \wtil{\al}_2 + \wtil{y}_2 \wtil{\al}_1\right)/\left(1-4\wtil{\al}_1 \wtil{\al}_2\right)$. Note that when $\rho_0=0$, that is, when $\wtil{y}_1 \wtil{\al}_2 + \wtil{y}_2 \wtil{\al}_1=0$, since $\wtil{\al}_1,\wtil{\al}_2\neq 0$ and have the same sign, $\wtil{y}_1,\wtil{y}_2\neq 0$ and have opposite signs, so that $f(\rho_0)=f(0)=-4 \wtil{y}_1 \wtil{y}_2>0$. Therefore, it sufficed to show that $f(\rho_0)\geq 0$. This occurs when
\[
(\wtil{y}_1 \wtil{\al}_2 + \wtil{y}_2 \wtil{\al}_1)^2 + \wtil{y}_1 \wtil{y}_2 (1-4\wtil{\al}_1 \wtil{\al}_2)\geq 0.
\]
Since $1-4\wtil{\al}_1 \wtil{\al}_2<0$, the inequality clearly holds when $\wtil{y}_1 \wtil{y}_2\leq 0$. On the other hand, for $\wtil{y}_1 \wtil{y}_2>0$ we have
\begin{align*}
&(\wtil{y}_1 \wtil{\al}_2 + \wtil{y}_2 \wtil{\al}_1)^2 + \wtil{y}_1 \wtil{y}_2 (1-4\wtil{\al}_1 \wtil{\al}_2)=\\
& \qquad \qquad = (\wtil{y}_1^2 \wtil{\al}_2^2 + 2 \wtil{\al}_1 \wtil{\al}_2 \wtil{y}_1 \wtil{y}_2 + \wtil{y}_2^2 \wtil{\al}_1^2
+ \wtil{y}_1 \wtil{y}_2 - 4\wtil{\al}_1 \wtil{\al}_2 \wtil{y}_1 \wtil{y}_2)\\
& \qquad \qquad = (\wtil{y}_1 \wtil{\al}_2 - \wtil{y}_2 \wtil{\al}_1)^2 + \wtil{y}_1 \wtil{y}_2> 0.
\end{align*}
Hence $y\in\fY_{\whatA}$ when $\wtil{y}_1\neq 0$ or $\wtil{y}_2\neq 0$.

Next take $y\in\BR^m$ with $\wtil{y}_1=\wtil{y}_2=0$, so that also $\wtil{y}=0$, but not all summands in $\wtil{y}_1$ and $\wtil{y}_2$ are zero. Such $y$ exist and we now show they are not in $\fY_{\whatA}$. Assume to the contrary that there exist $x\in\BR^n$ and $z\in\BR^q$ so that
\[
y=\whatA (z\otimes x) = \left(I_m+ \al \OneVec_m^T\right)E (z\otimes x) + \al\OneVec_{nq}^T (z\otimes x).
\]
Set
\[
y^\circ:=E(z\otimes x) \ands \la:=\OneVec_{nq}^T (z\otimes x)\quad \mbox{so that}\quad
y=\left(I_m+ \al \OneVec_m^T\right)y^\circ + \al\la.
\]
To use Lemma \ref{L:Case (iii) prep} we need to show that $x_r$ and $z_s$ are non-zero. Following the computations in the first part of the proof we have
\[
y^\circ = y + \frac{\wtil{y} -\la}{1-\wtil{\al}}\al = y - \frac{\la}{1-\wtil{\al}}\al.
\]

First consider the case where $\la=0$. Then $y^\circ = y$ and since not all summands in $\wtil{y}_1$ and $\wtil{y}_2$ are zero it follows that $x_r$ and $z_s$ are non-zero. Thus Lemma \ref{L:Case (iii) prep} applies, including \eqref{factForm} which leads to \eqref{laId2} that further specifies to
\[
0=\la=\left(\frac{0}{z_s} + x_r \right)
\left(\frac{0}{x_r} + z_s\right)= x_r z_s.
\]
However, this identity cannot hold since $x_r,z_s\neq 0$, hence $\la$ cannot be zero.

Since $\la\neq 0$, for $i=1,2$ we have $\wtil{y}^\circ_i=\wtil{y}_i -\frac{\wtil{y}-\la}{1-\wtil{\al}}\wtil{\al}_i =-\frac{\la}{1-\wtil{\al}}\wtil{\al}_i\neq 0$, because $\wtil{\al}_i\neq 0$. Thus, not all summands in $\wtil{y}^\circ_1$ and $\wtil{y}^\circ_2$ are zero so that we again obtain that $x_r,z_s\neq 0$ and that Lemma \ref{L:Case (iii) prep} applies. Then, as shown above, via \eqref{factForm} we obtain that \eqref{laId2} holds which implies $z_s x_r$ is a non-zero solution to \eqref{rhoId1}. This in turn means that $f(\rho)$ in \eqref{f} must attain a positive value or zero at a number $\rho\neq 0$. However, in the case considered here $f(\rho)=(1-4\wtil{\al}_1 \wtil{\al}_2)\rho^2$ which is zero at $\rho=0$ and negative for all other values of $\rho$. Hence $y\not\in\fY_{\whatA}$.

Finally, it remains to consider the case where $\wtil{y}_1=\wtil{y}_2=0$ and $y_1=\cdots =y_{q-1}=0$ or $y_q=\cdots =y_{m}=0$. Say $y_1=\cdots =y_{q-1}=0$, the other case goes similarly. Although $\fE_2$ in $\eqref{fE2}$ is not equal to $\BR^m$ in this case, we claim that $y\in \fE_2$. Indeed, take $x_r=0$ and $z_s$ so that $x_k$ for $k\neq r$ are all determined by $y=E(z\otimes x)$, but the numbers $z_k$ for $k\neq s$ are free to be chosen, so that we can easily arrange $\sum_{k=1}^q z_k=0$ leading to $\OneVec_{nq}^T (z\otimes x)=0$. Hence $y\in \fE_2$, as claimed. Note further that $y=\left(I_m-\al\OneVec_m^T\right)y$, since $\OneVec_m^T y=\wtil{y}=0$. Therefore, for $x$ and $z$ as above we have
\[
\whatA (z\otimes x)= \left(I_m-\al\OneVec_m^T\right)E(z\otimes x) +\al \OneVec_{nq}^T (z\otimes x)
= \left(I_m-\al\OneVec_m^T\right) y =y.
\]
Hence $y\in\fY_{\whatA}$.
\end{proof}

Although $\fY_{\whatA}\neq \BR^m$ in case (iii) with $\BF=\BR$ and $4\wtil{\al}_1 \wtil{\al}_2> 1$, the explicit computation of $\fY_{\whatA}$ in this case still enables us to prove that positivity and complete positivity coincide in this case.

\begin{lemma}\label{L:Case (iii) inv2}
Assume $L$ is as in case (iii) of Lemma \ref{L:fE2=Fm} and $\sum_{k=1}^m \alpha_k \neq1$. Then $\cL$ is positive if and only if $\cL$ is completely positive.
\end{lemma}

\begin{proof}
By Lemma \ref{L:Case (iii) inv} we need only consider the case where $\BF=\BR$ and $4\wtil{\al}_1 \wtil{\al}_2> 1$. We need to show that $y^T \BH y\geq 0$ for all $y\in \fY_{\whatA}$ implies $\BH\geq 0$, where $\BH$ is the Hill matrix associated with the choice of $L_1,\ldots,L_m$. Let $y\not\in \fY_{\whatA}$. Then $\wtil{y},\wtil{y}_1,\wtil{y}_2$ in \eqref{tilaly} satisfy $\wtil{y}=\wtil{y}_1=\wtil{y}_2=0$. Then for $\mu\neq 0$ we have $y+\mu(e_1+e_m)\in \fY_{\whatA}$ so that
\begin{align*}
&y^T \BH y + \mu \left(\left(e_1+e_m\right)^T\BH y+ y^T \BH \left(e_1+e_m\right)\right) + \mu^2 \left(e_1+e_m\right)^T \BH \left(e_1+e_m\right)=\\
&\qquad\qquad\qquad =  \left(y+\mu\left(e_1+e_m\right)\right)^T\BH \left(y+\mu\left(e_1+e_m\right)\right) \geq 0.
\end{align*}
Since this inequality holds for each $\mu\neq 0$, it must also hold for $\mu=0$, so that also $y^T \BH y\geq 0$ for $y\not \in \fY_{\whatA}$ and we can conclude that $\BH\geq 0$.
\end{proof}

\subsection{The case where $L_{ij}=L_0$ for all $(i,j)\neq (i_k,j_k)$ and $\sum_{k=1}^m \al_k= 1$.}

Next we consider the case where $\sum_{k=1}^m \al_k= 1$, so that the matrix $I_m-\al \OneVec_{m}^T$ is not invertible. We distinguish between the case where $\fE_2$ in \eqref{fE2} is equal to $\BF^m$, so that (C1) and one of (C2.1)--(C2.7) holds, and the case where (C1) holds but $\fE_2\neq \BF^m$, so that $L$ must be as in (i), (ii) or (iii) in Lemma \ref{L:fE2=Fm}. For $L$ as in (i) and (ii) the result is already proved in Lemma \ref{L:Case (i) and (ii)}. The remaining cases are proved in a series of lemmas; in some of the proofs the condition $\sum_{k=1}^m \al_k= 1$ is not required, while in others it is.

\begin{lemma}\label{L:C2.1&C2.2}
Assume condition (C1) in \eqref{C1} together with (C2.1) or (C2.2) holds. Then $\fY_{\whatA}=\BF^m$.
\end{lemma}

\begin{proof}[\bf Proof]
We give a proof for the case where (C1) and (C2.2) holds. The case where (C2.2) is replaced by (C2.1) is proved analogously.

By assumption there exists a $j_0$ such that $j_k\neq j_0$ for all $k$. This implies that $E\left(b e_{j_0}\otimes x\right)=0$ for all $b\in\BF$ and all $x\in\BF^n$. Let $y\in \BF^m$. By Lemma \ref{L:fE1=Fm} there exist $z\in\BF^q$ and $x\in\BF^n$ so that $y= E(z\otimes x)$. Then, for any $b\in\BF$ we have
\begin{align*}
\whatA \left(\left(z + b e_{j_0}\right)\otimes x\right) &= E\left(\left(z + b e_{j_0}\right)\otimes x\right)+\al\left(\OneVec_{nq}^T-\OneVec_{m}^T E\right) \left(\left(z + b e_{j_0}\right)\otimes x\right)\\
&=y +\al \left(\OneVec_{nq}^T\left(\left(z + b e_{j_0}\right)\otimes x\right) - \OneVec_{m}^T y\right)\\
&= y +\al \mbox{$\left(\left(\sum_{j=1}^q z_j\right)\left(\sum_{i=1}^n x_i\right)+b\left(\sum_{i=1}^n x_i\right) -\left(\sum_{k=1}^m y_k\right)\right)$}.
\end{align*}
To prove our claim we will show that for every $y$ we can find $x$, $z$ and $b$ so that the second term becomes 0. Note that this can always be arranged if we have $y= E(z\otimes x)$ with $\sum_{i=1}^n x_i \neq 0$. In that case, simply take
\begin{equation}\label{b}
b=\frac{\sum_{k=1}^m y_k - \left(\sum_{j=1}^q z_j\right)\left(\sum_{i=1}^n x_i\right)}{\left(\sum_{i=1}^n x_i\right)}.
\end{equation}
First we consider a case where this may not happen, after which we show that in all remaining cases it is possible to have $\sum_{i=1}^n x_i \neq 0$.

\paragraph{\bf Case 1}
Assume that $i_k\neq i_l$ for all $k\neq l$. In this case $x$ and $z$ with $y= E(z \otimes x)$ can be obtained by setting $z_j=1$ for all $j$, $x_{i_k}=y_k$ and $x_i=0$ for $i\ne i_1, \ldots, i_m.$ Then $\sum_{k=1}^m y_k=\sum_{i=1}^n x_i$. Then, in case $\sum_{k=1}^m y_k\neq 0$, also $\sum_{i=1}^n x_i\neq 0$ and we can take $b$ as in \eqref{b}. In case $\sum_{k=1}^m y_k=0$, we also have $\sum_{i=1}^n x_i= 0$ and thus
\[
\left(\sum_{j=1}^q z_j\right)\left(\sum_{i=1}^n x_i\right)+b\left(\sum_{i=1}^n x_i\right) -\left(\sum_{k=1}^m y_k\right) = 0
\]
irrespectively of the choice of $b$. Hence both for $\sum_{k=1}^m y_k\neq 0$ and $\sum_{k=1}^m y_k=0$ we can find $z$ and $x$ with $y=\whatA (z\otimes x)$.

\paragraph{\bf Case 2}
Assume that $j_k\neq j_l$ for all $k\neq l$. In that case take $x_i=1$ for all $i$, $z_{j_k}=y_k$ for all $k$ and $z_j=0$ for $j\neq j_1,\ldots,j_m$. Then $y=E(z\otimes x)$ and $\sum_{i=1}^n x_i=n\neq 0$ so that we can take $b$ as in \eqref{b}.

\paragraph{\bf Case 3}
Assume we are neither in Case 1 nor in Case 2. Since (C1) holds this implies that, after possibly rearranging indices of $L_1,\ldots,L_m$ as well as block columns and rows in $L$ (see Remark \ref{R:Reorder}), there exist $s,r>1$ so that $i_k=k$ and $j_k=1$ for $k=1,\ldots,r$, $i_{r+l}=r+1$ and $j_{r+l}=l+1$ for $l=1,\ldots,s$ and $i_{r+s+1}=r+2$, $j_{r+s+1}=s+2$, while $i_p\neq 1,\ldots, r+2$ and $j_p\neq 1,\ldots, s+2$ for $p> r+s+1$. Let $y\in\BF^m$ be given as $y=E(z \otimes x)$ for some $z \in \BF^q$ and $x \in \BF^n$. Then $y_k= x_k z_1$ for $k=1,\ldots,r$, $y_{r+l}= x_{r+1} z_{1+l}$ for $l=1,\ldots,s$ and $y_{r+s+1}=x_{r+2}z_{s+2}$, while $z_1,\ldots,z_{s+2},x_1,\ldots,x_{r+2}$ do not occur in the factorisations $y_p=x_{i_p}z_{j_p}$ for $p>r+s+1$. As we did in the proof of Lemma \ref{L:fE2=Fm} for (C2.7), we will adjust $z$ and $x$ while maintaining $y=E(z \otimes x)$ and in this case arranging $\sum_{i=1}^n x_i \neq 0$. Redefine $z_1:=1$, $x_i:=y_i$ for $i=1,\ldots,r$, select $0\neq x_{r+1} \neq -\sum_{i={r+2}}^n x_i - \sum_{k=1}^r y_k$ and redefine $z_{1+j}:=y_{r+j}/ x_{r+1}$ for $j=1,\ldots,s$. In that case we still have $y_k= x_{i_k}y_{j_k}$ for $k=1,\ldots,s+r+1$ while the other identities remain unaffected and also
\[
\sum_{i=1}^n x_i = \sum_{k=1}^r y_k + x_{r+1} + \sum_{i=r+2}^n x_i \neq 0.
\]
Hence, for these redefined $x$ and $z$ we can take $b$ as in \eqref{b}.
\end{proof}

\begin{lemma}\label{L:C2.5&C2.6}
Assume condition (C1) in \eqref{C1} together with (C2.5) or (C2.6) holds. Assume also that $\sum_{k=1}^m\al_k = 1$. Then $\fY_{\whatA}=\BF^m$.
\end{lemma}

\begin{proof}[\bf Proof]
We give a proof of the case where (C1) and (C2.6) hold. For (C1) and (C2.5) a similar argument applies. Hence, assume (C1) and (C2.6) hold. Note that (C2.6) in fact implies (C1). Also, we have $m\leq q$. However, if $m<q$, then also (C2.2) holds and we obtain that $\fY_{\whatA}=\BF^m$ from Lemma \ref{L:C2.1&C2.2}. Hence we may assume that $m=q$ and, possibly after relabeling $L_1,\ldots,L_m$ and rearranging block rows (see Remark \ref{R:Reorder}), that $j_k=k$ for $k=1,\ldots,m$. Note that we can write
\[
\whatA (z\otimes x) = \whatA (I_q \otimes x) z.
\]
Hence, we obtain that $\fY_{\whatA}=\BF^{m}$ in case we can find a $x\in\BF^n$ so that $\whatA (I_q \otimes x)\in\BF^{m\times m}$ has full row rank, or, equivalently, is invertible.

We have $E (I_q \otimes x)=\diag(x_{i_1},\ldots, x_{i_m})=E^T(I_q \otimes x)^T$, which is invertible precisely when $x_{i_k}\neq 0$ for all $k$, and
\begin{align*}
v^T\left(I_q \otimes x\right) & =\sum_{i=1}^n x_i \OneVec_q^T -\OneVec_q^T E \left(I_q \otimes x\right).
\end{align*}
Restricting to $x_{i_k}\neq 0$ for all $k$ and applying the Sherman–Morrison formula \cite{H89} to
\[
\whatA \left(I_q \otimes x\right)=E\left(I_q \otimes x\right)+ \al v^T \left(I_q \otimes x\right)
\]
it follows that $\whatA (I_q \otimes x)$ is invertible if and only if
\begin{align*}
-1 & \neq v^T \left(I_q \otimes x\right) \left(E\left(I_q \otimes x\right)\right)^{-1} \al
= \sum_{i=1}^n x_i \OneVec_q^T \left(E\left(I_q \otimes x\right)\right)^{-1} \al  -\OneVec_q^T  \al \\
&= \sum_{i=1}^n x_i \mat{x_{i_1}^{-1} & \cdots & x_{i_m}^{-1}} \al - \sum_{k=1}^m \al_k
= \sum_{i=1}^n x_i \mat{x_{i_1}^{-1} & \cdots & x_{i_m}^{-1}} \al -1.
\end{align*}
Thus we need to find an $x\in \BF^n$ with $x_{i_k}\neq 0$ for all $k$ so that
\[
\sum_{i=1}^n x_i \mat{x_{i_1}^{-1} & \cdots & x_{i_m}^{-1}} \al \neq 0.
\]
This can simply be done by selecting $x_i=1$ for all $i$, giving $n$ on the left hand side, since by assumption $\sum_{k=1}^m\al_k = 1$. Hence we can choose $x$ so that $\whatA(I_q \otimes x)$ is invertible, which implies $\fY_{\whatA}=\BF^{m}$.
\end{proof}

\begin{lemma}\label{L:C2.3&C2.4}
Assume condition (C1) in \eqref{C1} together with (C2.3) or (C2.4) holds.  Assume also that $\sum_{k=1}^m\al_k = 1$. Then $\fY_{\whatA}=\BF^m$.
\end{lemma}

\begin{proof}[\bf Proof]
We give a proof of the case where (C1) and (C2.4) hold. For (C1) and (C2.3) a similar argument applies. Hence, assume (C1) and (C2.4) hold. Since the claim is proved for (C1) together with (C2.5) in Lemma \ref{L:C2.5&C2.6} we will exclude condition (C2.5).

As in the proof of Lemma \ref{L:fE2=Fm} we note that there exist $r,s>1$ (but $r+s \le n$) so that, possibly after rearranging the indices of $L_1,\ldots,L_m$ as well as block columns and rows in $L$ (see Remark \ref{R:Reorder}), $j_k=1$ and $i_k=k$ for $k=1,\ldots,r$ and $j_{r+l}=2$ and $i_{r+l}=r+l$ for $l=1,\ldots,s$, while $j_k\neq 1,2$ and $i_k\neq 1,\ldots, r+s$ whenever $k>r+s$. Furthermore, since (C2.5) is excluded we may assume that there exists a $p>1$ such that $i_{r+s+i}=i_{r+s+1}$ and $j_{r+s+i}=2+i$ for $i=1,\ldots,p$, while $j_k\neq 1,\ldots, 2+p$ and $i_k\neq 1,\ldots, r+s+1$ whenever $k>r+s+p$.

Now fix a $y^\circ\in\BF^m$. Using Lemma \ref{L:Ipert} and \eqref{fE1} it follows that there exist $y\in\BF^m$ and  $\la\in\BF$ so that $y^\circ = \left(I_m - \alpha \OneVec_m^T\right)y + \al \la$ and for this $y$ there exist $z \in \BF^q$ and $x \in \BF^n$ so that $y=E(z \otimes x)$. Then $y_k= x_k z_1$ for $k=1,\ldots,r$ and $y_{r+k}= x_{r+k} z_2$ for $k=1,\ldots,s$ and $y_{r+s+k}= x_{r+s+1} z_{2+k}$ for $k=1,\ldots,p$, while $z_1,\ldots,z_{2+p},x_1,\ldots,x_{r+s+1}$ do not occur in the factorisations $y_k=x_{i_k}z_{j_k}$ for $k>r+s+p$. We claim that we can adjust $z$ and $x$ in such a way that still $y=E(z \otimes x)$ and in addition $\OneVec_{nq}^T(z \otimes x)=\la$. Once this is achieved, we obtain that
\[
y^\circ = \left(I_m - \alpha \OneVec_m^T\right)y + \al \la = \left(I_m - \alpha \OneVec_m^T\right)E(z \otimes x) + \al \OneVec_{nq}^T(z \otimes x)=\whatA (z \otimes x),
\]
and we may conclude that $\fY_{\whatA}=\BF^m$, since $y^\circ$ was chosen arbitrarily.

To see that $z$ and $x$ can be adjusted in the desired way, we will only modify  $z_1,z_2,x_{r+s+1}$, subject to $z_1,z_2,x_{r+s+1}\neq 0$, so that $x_k$ for $k=1,\ldots,r+s$ and $z_k$ for $k=3,\ldots,2+p$ are fixed by the above equations while the other values of $x_i$ and $z_j$ remain unchanged. This guarantees that $y=E(z \otimes x)$ remains true. Set
\begin{align*}
\wtil{y}_1:=\sum_{k=1}^r y_k,\quad \wtil{y}_2:=\sum_{k=r+1}^{r+s} y_k, \quad \wtil{y}_3:=\sum_{k=r+s+1}^{r+s+p} y_k,\\
\la_1:=\sum_{i=r+s+2}^n x_{i},\quad
\la_2:=\sum_{j=2+p+1}^q z_{j}.
\end{align*}
Then
\begin{align*}
\sum_{i=1}^n x_{i} &= \sum_{i=1}^{r+s} x_{i} + x_{r+s+1} + \sum_{i=r+s+2}^{n} x_{i}
= \sum_{i=1}^{r} \frac{y_i}{z_1} + \sum_{i=r+1}^{r+s} \frac{y_i}{z_2} + x_{r+s+1} +\la_1\\
&=\frac{\wtil{y}_1}{z_1}+ \frac{\wtil{y}_2}{z_2}+ x_{r+s+1} +\la_1,
\end{align*}
and via a similar computation
\[
\sum_{j=1}^q z_{j}=z_1+z_2 +\frac{\wtil{y}_3}{x_{r+s+1}} +\la_2.
\]
Thus the objective is to find $z_1,z_2,x_{r+s+1}\neq 0$ so that
\[
\left(\frac{\wtil{y}_1}{z_1}+ \frac{\wtil{y}_2}{z_2}+ x_{r+s+1} +\la_1\right)
\left(z_1+z_2 +\frac{\wtil{y}_3}{x_{r+s+1}} +\la_2\right)=\la.
\]
In case one of $\wtil{y}_1,\wtil{y}_2,\wtil{y}_3$ is zero, this ensures that one of the variables only appears in one of the two factors, after which the objective is easy to achieve.  When $\la=0$, $z_1$ and $z_2$ can be chosen (with $z_1,z_2\neq 0$) so that the second factor becomes zero, and the identity is satisfied. Thus, assume $\la,\wtil{y}_1,\wtil{y}_2,\wtil{y}_3\neq 0$.

Now we select $x_{r+s+1}$ so that the first factor becomes equal to 1, i.e., $x_{r+s+1}=1-\la_1 -\frac{\wtil{y}_1}{z_1} - \frac{\wtil{y}_2}{z_2}$, and insert this into the second factor so that we arrive at
\[
\la=z_1+z_2 +\frac{\wtil{y}_3}{x_{r+s+1}} +\la_2= z_1+z_2 +\frac{ \wtil{y}_3 z_1z_2}{(1-\la_1)z_1z_2 -\wtil{y}_1 z_2 - \wtil{y}_2 z_1} +\la_2,
\]
keeping in mind the conditions $z_1,z_2\neq 0$, $\la_1 +\frac{\wtil{y}_1}{z_1} + \frac{\wtil{y}_2}{z_2}\neq 1$. Multiplying with $(1-\la_1)z_1z_2 -\wtil{y}_1 z_2 - \wtil{y}_2 z_1$ on both sides this leads to a second order polynomial equation in $z_1$ with parameters depending on $z_2$:
\[
a(z_2) z_1^2  + b(z_2) z_1  + c(z_2)=0
\]
with
\begin{align*}
a(z_2)=z_2\left(1-\la_1\right)&-\wtil{y}_2,\quad c(z_2)= z_2 \left(\la -\la_2\right)\wtil{y}_1-z_2^2 \wtil{y}_1,\\
b(z_2)=z_2^2\left(1-\la_1\right) + z_2&\left(\wtil{y}_3-\wtil{y}_1-\wtil{y}_2 -\left(1-\la_1\right)\left(\la-\la_2 \right)\right)+\wtil{y}_2\left(\la-\la_2\right).
\end{align*}

First assume $\la_1 \neq 1$. In that case $b(z_2)^2$ is a polynomial in $z_2$ of order 4 while $4 a(z_2) c(z_2)$ is of order 3. Hence, taking $z_2$ large enough, avoiding $\frac{\wtil{y}_2}{1-\la_1}$ and $\la-\la_2$ so that $a(z_2)$ and $c(z_2)$ are non-zero, we end up with two non-zero solutions for $z_1$, which are real in case $\BF=\BR$, so that for at least one of the solutions for $z_1$ we will have $\la_1 +\frac{\wtil{y}_1}{z_1} + \frac{\wtil{y}_2}{z_2}\neq 1$.

Now assume $\la_1 =1$. Then $c\left(z_2\right)$ is unchanged and
\begin{align*}
a\left(z_2\right)=-\wtil{y}_2,\quad
b\left(z_2\right)= z_2\left(\wtil{y}_3-\wtil{y}_1-\wtil{y}_2\right)+\wtil{y}_2\left(\la-\la_2\right).
\end{align*}
In this case $b\left(z_2\right)^2-4a\left(z_2\right)c\left(z_2\right)$ is a polynomial of degree at most two with constant term $\wtil{y}_2^2\left(\la-\la_2\right)^2$. If $\la\neq\la_2$, this constant term is positive and taking $z_2\neq 0$ small enough, avoiding $\la-\la_2$ so that again $a\left(z_2\right)$ and $c\left(z_2\right)$ are non-zero, we again arrive at two non-zero solutions for $z_1$ (being real in case $\BF=\BR$) so that for one of them we have $\la_1 +\frac{\wtil{y}_1}{z_1} + \frac{\wtil{y}_2}{z_2}\neq 1$.

Finally, we consider the case where $\la_1=1$ and $\la_2=\la$. This means, since both $\la_1, \la_2 \neq 0$, that $m > r+s+p$ and that $q>p+2$ and $n>r+s+1.$ We show that in this case $x$ and $z$ can be modified in such a way that $y^\circ = \whatA (z\otimes x)$ remains true while also $\la_1 \ne 1$ or $\la_2\ne\la$ holds, so that the constructions of the previous paragraphs can be applied.

By rearranging some of the indices of $L_{r+p+s},\ldots,L_m$ and reordering some of the block rows and columns (see Remark \ref{R:Reorder}) we can arrange to have $y_m=x_nz_q$ as the last equation in $y=E(z\otimes x)$. Since (C1) holds, we know either $z_q$ or $x_n$ occurs only in a single factorisation $y_k=x_{i_k}z_{j_k}$ (hence with $k=m$). Assume $x_n$ occurs only in $y_m=x_nz_q$. Take $a\in\BF$ so that $a\neq 0$ and $a\neq -z_q$. Adjust $z_q$ to $z_q+a$ and redefine $x_{i_k}:=\frac{y_{k}}{z_q+a}$ whenever $j_k=q$. Note that $x_{j_k}$ cannot occur in another factorisation $y_k=x_{i_k}z_{j_k}$ when $j_k=q$, since either $k=m$ (so that $i_k=n$) or since there is more than one $k$ with $j_k=q$ and (C1) is in place. Hence, with these modifications $y=E(z\otimes x)$ still holds while $\la_2$ is adjusted to $\la_2+a \neq \la$. A similar argument allows one to adjust $x$ and $z$ to achieve $\la_1\neq 1$ in case $z_q$ occurs only in the factorisation $y_m=x_nz_q$.
\end{proof}

\begin{lemma}\label{L:C2.7}
Assume condition (C1) in \eqref{C1} together with (C2.7) holds.  Assume also that $\sum_{k=1}^m\al_k = 1$. Then $\fY_{\whatA}=\BF^m$.
\end{lemma}

\begin{proof}[\bf Proof]
Since all other cases are covered, we may assume the conditions (C2.1) to (C2.6) do not hold. By excluding the case $m=1$, reasoning as for (C2.4) again after possibly rearranging the indices of $L_1,\ldots,L_m$ as well as block columns and rows in $L$ (see Remark \ref{R:Reorder}), it follows that there exist $r,s>1$ (with $r+s+1\leq m$) so that $i_k=k$ and $j_k=1$ for $k=1,\ldots,r$, $i_{r+l}=r+1$ and $j_{r+l}=l+1$ for $l=1,\ldots,s$ and $i_{r+s+1}=r+2$, $j_{r+s+1}=s+2$, while $i_p\neq 1,\ldots, r+2$ and $j_p\neq 1,\ldots, s+2$ for $p> r+s+1$. Now fix a $y^\circ\in\BF^m$. Using Lemma \ref{L:Ipert} and \eqref{fE1} it follows that there exist $y\in\BF^m$ and  $\la\in\BF$ so that $y^\circ = \left(I_m - \alpha \OneVec_m^T\right)y + \al \la$ and for this $y$ there exist $z \in \BF^q$ and $x \in \BF^n$ so that $y=E(z \otimes x)$. Then $y_k= x_k z_1$ for $k=1,\ldots,r$, $y_{r+l}= x_{r+1} z_{1+l}$ for $l=1,\ldots,s$ and $y_{r+s+1}=x_{r+2}z_{s+2}$, while $z_1,\ldots,z_{s+2},x_1,\ldots,x_{r+2}$ do not occur in the factorisations $y_p=x_{i_p}z_{j_p}$ for $p>r+s+1$. We claim that we can adjust $z$ and $x$ in such a way that still $y=E(z \otimes x)$ and in addition $\OneVec_{nq}^T(z \otimes x)=\la$. Once this is achieved, we obtain that
\[
y^\circ = \left(I_m - \alpha \OneVec_m^T\right)y + \al \la = \left(I_m - \alpha \OneVec_m^T\right)E(z \otimes x) + \al \OneVec_{nq}^T(z \otimes x)=\whatA (z \otimes x),
\]
and we may conclude that $\fY_{\whatA}=\BF^m$, since $y^\circ$ was chosen arbitrarily.

To see that $z$ and $x$ can be adjusted in the desired way, we will only modify  $z_1,x_{r+1}$ and one of $x_{r+2}$ or $z_{s+2}$, subject to $z_1,x_{r+1},x_{r+2},z_{s+2}\neq 0$, so that $x_k$ for $k=1,\ldots,r$, $z_k$ for $k=2,\ldots,s+1$ and either $x_{r+2}$ or $z_{s+2}$ are fixed by the above equations while the other values of $x_i$ and $z_j$ remain unchanged. This guarantees that $y=E(z \otimes x)$ remains true. Set
\begin{align*}
\wtil{y}_1:=\sum_{k=1}^r y_k,\quad \wtil{y}_2:=\sum_{k=r+1}^{r+s} y_k, \quad
\la_1:=\sum_{i=r+3}^n x_{i},\quad
\la_2:=\sum_{j=s+3}^q z_{j}.
\end{align*}
Then
\begin{align*}
\sum_{i=1}^n x_{i} &= \sum_{i=1}^{r}x_i+ x_{r+1} + x_{r+2} + \sum_{i=r+3}^{n} x_{i}
= \sum_{i=1}^{r} \frac{y_i}{z_1} + x_{r+1}+\frac{y_{r+s+1}}{z_{s+2}}  +\la_1\\
&=\frac{\wtil{y}_1}{z_1}+ x_{r+1} + \frac{y_{r+s+1}}{z_{s+2}} +\la_1,
\end{align*}
and via a similar computation
\[
\sum_{j=1}^q z_{j}=z_1+z_{s+2} +\frac{\wtil{y}_2}{x_{r+1}} +\la_2.
\]
Thus the objective is to find $z_1,x_{r+1},z_{s+2}\neq 0$ so that
\[
\left(\frac{\wtil{y}_1}{z_1}+ \frac{y_{r+s+1}}{z_{s+2}}+ x_{r+1} +\la_1\right)
\left(z_1+z_{s+2} +\frac{\wtil{y}_2}{x_{r+1}} +\la_2\right)=\la.
\] In case one of $\wtil{y}_1,\wtil{y}_2,y_{r+s+1}$ is zero, this ensures that one of the variables only appears in one of the two factors, after which the objective is easy to achieve.  When $\la=0$, $z_1$ and $z_{s+2}$ can be chosen (with $z_1,z_{s+2}\neq 0$) so that the second factor becomes zero, and the identity is satisfied. Thus, assume $\la,\wtil{y}_1,\wtil{y}_2,y_{r+s+1}\neq 0$.

Now we select $x_{r+1}$ so that the first factor becomes equal to 1, i.e., $x_{r+1}=1-\la_1 -\frac{\wtil{y}_1}{z_1} - \frac{y_{r+s+1}}{z_{s+2}}$, and insert this into the second factor so that we arrive at
\[
\la=z_1+z_{s+2} +\frac{\wtil{y}_2}{x_{r+1}} +\la_2= z_1+z_{s+2} +\frac{ \wtil{y}_2 z_1z_{s+2}}{(1-\la_1)z_1z_{s+2} -\wtil{y}_1 z_{s+2} - y_{r+s+1} z_1} +\la_2,
\]
keeping in mind the conditions $z_1,z_{s+2}\neq 0$, $\la_1 +\frac{\wtil{y}_1}{z_1} + \frac{y_{r+s+1}}{z_{s+2}}\neq 1$. Multiplying with $(1-\la_1)z_1z_{s+2} -\wtil{y}_1 z_{s+2} - y_{r+s+1}z_1$ on both sides this leads to a second order polynomial equation in $z_1$ with parameters depending on $z_{s+2}$:
\[
a\left(z_{s+2}\right) z_1^2  + b\left(z_{s+2}\right) z_1  + c\left(z_{s+2}\right)=0
\]
with
\begin{align*}
a\left(z_{s+2}\right)&=z_{s+2}\left(1-\la_1\right)-y_{r+s+1},\quad c\left(z_{s+2}\right)= z_{s+2} \left(\la -\la_2\right)\wtil{y}_1-z_{s+2}^2 \wtil{y}_1,\\
b\left(z_{s+2}\right)&=z_{s+2}^2\left(1-\la_1\right) + z_{s+2}\left(\wtil{y}_2-\wtil{y}_1-y_{r+s+1} -\left(1-\la_1\right)\left(\la-\la_2\right)\right)+\\
&\qquad\qquad\qquad+y_{r+s+1}\left(\la-\la_2\right).
\end{align*}
First assume $\la_1 \neq 1$. In that case $b\left(z_{s+2}\right)^2$ is a polynomial in $z_{s+2}$ of order 4 while $4 a\left(z_{s+2}\right) c\left(z_{s+2}\right)$ is of order 3. Hence, taking $z_{s+2}$ large enough, avoiding $\frac{y_{r+s+1}}{1-\la_1}$ and $\frac{\wtil{y}_{1}(\la-\la_2)}{\wtil{y}_1}$ so that $a\left(z_{s+2}\right)$ and $c\left(z_{s+2}\right)$ are non-zero, we end up with two non-zero solutions for $z_1$, which are real in case $\BF=\BR$, so that at least one of the solutions for $z_1$ will satisfy $\la_1 +\frac{\wtil{y}_1}{z_1} + \frac{y_{r+s+1}}{z_{s+2}}\neq 1$.

Now assume $\la_1 =1$. Then $c\left(z_{s+2}\right)$ is unchanged and
\begin{align*}
a\left(z_{s+2}\right)=-y_{r+s+1},\quad
b\left(z_{s+2}\right)= z_{s+2}\left(\wtil{y}_2-\wtil{y}_1-y_{r+s+1}\right)+y_{r+s+1}\left(\la-\la_2\right).
\end{align*}
In this case $b\left(z_{s+2}\right)^2-4a\left(z_{s+2}\right)c\left(z_{s+2}\right)$ is a polynomial of degree at most two with constant term $y_{r+s+1}^2\left(\la-\la_2\right)^2$. If $\la\neq\la_2$, this constant term is positive and taking $z_{s+2}\neq 0$ small enough, avoiding $\la-\la_2$ so that again $a\left(z_{s+2}\right)$ and $c\left(z_{s+2}\right)$ are non-zero, we again arrive at two non-zero solutions for $z_1$ (being real in case $\BF=\BR$) so that one of them satisfies $\la_1 +\frac{\wtil{y}_1}{z_1} + \frac{y_{r+s+1}}{z_{s+2}}\neq 1$.

The case that remains is when $\la_1=1$ and $\la_2=\la$, which implies that $m > r+s+1$, $q>s+2$ and $n>r+2$, since $\la_1, \la_2 \neq 0$. By a similar argument as used in the last paragraph of the proof of Lemma \ref{L:C2.3&C2.4} the vectors $x$ and $z$ can be modified to obtain $\la_1\neq 1$ or $\la_2\neq \la$, so that the constructions of the previous paragraphs can be applied.
%
\end{proof}

Lastly we prove for case (iii) in Lemma \ref{L:fE2=Fm} that $\fY_{\whatA}=\BF^m$ when $\sum_{k=1}^m \alpha_k =1$.

\begin{lemma}\label{L:Case (iii) noninv}
Assume $L$ is as in case (iii) of Lemma \ref{L:fE2=Fm} and $\sum_{k=1}^m \alpha_k =1$. Then $\fY_{\whatA}=\BF^m$.
\end{lemma}

\begin{proof}[\bf Proof]
Assume $L$ has the form as in case (iii) in Lemma \ref{L:fE2=Fm} and $\sum_{k=1}^m \alpha_k=1$.
Fix a $y\in\BF^m$. Using Lemma \ref{L:Ipert} it follows that there exist $y^\circ\in\BF^m$ and  $\la\in\BF$ so that $y = \left(I_m - \alpha \OneVec_m^T\right)y^\circ + \al \la$. Then by Lemma \ref{L:Case (iii) prep} for all $x_r,z_s\neq 0$ there exist unique vectors $x$ and $z$, with $x_r$ and $z_s$ in the indicated positions of $x$ and $z$, respectively, so that $y^\circ=E(z \otimes x)$ and with $\OneVec_{nq}^T (z \otimes x)$ as in \eqref{factForm}. In the remainder of the proof we show that $x_r$ and $z_s$ can be chosen in such a way that $\OneVec_{nq}^T (z \otimes x)=\la$. Once that is established, it follows that
\[
y = (I_m - \alpha \OneVec_m^T)y^\circ + \al \la = (I_m - \alpha \OneVec_m^T)E(z \otimes x) + \al \OneVec_{nq}^T(z \otimes x)=\whatA (z \otimes x),
\]
and the proof is complete because $y\in\BF^m$ was chosen arbitrarily.

Since $\sum_{k=1}^m\alpha_k =1$, we have $\kr I_m - \alpha \OneVec_m^T=\operatorname{span}\{\alpha\}$, by Lemma \ref{L:Ipert}. This implies that $y^\circ$ can be written as $y^\circ=y'+\rho \al$ with $y'\perp \al$ and $\rho\in\BF$ is another variable we can select arbitrarily. Set
\begin{align*}
\wtil{y}'_1:=\sum_{k=1}^{q-1} y'_k,\quad \wtil{y}'_2:=\sum_{k=q}^{m} y'_k,\quad
 \widetilde{\alpha}_1:=\sum_{k=1}^{q-1} \alpha_k,\quad \widetilde{\alpha}_2:=\sum_{k=q}^m\alpha_k,
\end{align*}
so that
\[
\wtil{y}_1^\circ:=\sum_{k=1}^{q-1} y^\circ_k = \wtil{y}'_1 + \rho \wtil{\al}_1 \ands
\wtil{y}_2^\circ:=\sum_{k=q}^{m} y^\circ_k = \wtil{y}'_2 + \rho \wtil{\al}_2.
\]
Hence, using \eqref{factForm} and the above formulas for $\wtil{y}_1^\circ$ and $\wtil{y}_2^\circ$ the equation to solve becomes
\begin{equation}\label{laId1}
\left(\frac{\wtil{y}'_2+\rho \wtil{\al}_2}{z_s} +x_r\right)
\left(\frac{\wtil{y}'_1+\rho \wtil{\al}_1}{x_r}+z_s\right)=
\la
\end{equation}
with variables $x_r,z_s,\rho\in\BF$ subject to $x_r,z_s\neq 0$. We show this equation can be solved by considering three cases.

\paragraph{\bf Case 1} Assume $\la=0$. Take $z_s=x_r=1$ and $\rho=\frac{-\left(\wtil{y}'_2+1\right)}{\wtil{\al}_2}$.

\paragraph{\bf Case 2} Assume $\wtil{y}'_1 \neq \la+\frac{\wtil{\al}_1 \wtil{y}'_2}{\wtil{\al}_2}$. Take $\rho=\frac{-\wtil{y}'_2}{\wtil{\al}_2}$. Then
\[
\wtil{y}'_2+\rho \wtil{\al}_2=0 \ands \wtil{y}'_1+\rho \wtil{\al}_1= \wtil{y}'_1-\frac{\wtil{\al}_1\wtil{y}'_2}{\wtil{\al}_2}\neq \la.
\]
Thus \eqref{laId1} holds with $x_r=1$ and $z_s=\la-\left(\wtil{y}'_1-\frac{\wtil{\al}_1\wtil{y}'_2}{\wtil{\al}_2}\right)\neq 0$.

\paragraph{\bf Case 3} Assume $\la\neq 0$ and $\wtil{y}'_1 =\la+\frac{\wtil{\al}_1 \wtil{y}'_2}{\wtil{\al}_2}$. Write $\rho$ as $\rho=\frac{\be -\wtil{y}'_2}{\wtil{\al}_2}$ with $\be\in\BF$ arbitrary. In this case \eqref{laId1} turns into
\[
\left(\frac{\be}{z_s} +x_r\right)
\left(\frac{\la +\be \frac{\wtil{\al}_1}{\wtil{\al}_2}}{x_r}+z_s\right)=\la.
\]
Set $z_s=1$ and multiply both sides with $x_r$. This yields the equation
\[
0=\left(\be +x_r\right)\left(\mbox{$\la +\be \frac{\wtil{\al}_1}{\wtil{\al}_2}$}+x_r\right)-\la x_r=x_r^2 + \be\left( 1+\mbox{$\frac{\wtil{\al}_1}{\wtil{\al}_2}$}\right)x_r+\be\left(\mbox{$\la +\be \frac{\wtil{\al}_1}{\wtil{\al}_2}$}\right).
\]

In case $\wtil{\al}_1=\wtil{\al}_2$, take $\be$ so that $\be(\la+\be)\neq 0$ (with $\be(\la+\be)< 0$ if $\BF=\BR$) to obtain a non-zero solution for $x_r$ (which is real in case if $\BF=\BR$).

In case $\wtil{\al}_1\neq \wtil{\al}_2$, for $\BF=\BC$ it is easy to select a $\be$ so that the equation has a non-zero solution, while for $\BF=\BR$ we need to choose $\be$ so that
\[
\be^2\left(1+\mbox{$\frac{\wtil{\al}_1}{\wtil{\al}_2}$}\right)^2 -4 \be\left(\mbox{$\la +\be \frac{\wtil{\al}_1}{\wtil{\al}_2}$}\right)=
\be^2\left(1-\mbox{$\frac{\wtil{\al}_1}{\wtil{\al}_2}$}\right)^2 -4 \be \la >0.
\]
This can easily be done by taking $\be$ large enough.
\end{proof}

\subsection{Proof of Theorem \ref{T:Main}}

In this section we merge the results from the previous sections to prove Theorem \ref{T:Main}. Throughout we assume condition (C1) in \eqref{C1intro} holds. Then by Lemma \ref{L:fE2=Fm} either one of the conditions (C2.1)--(C2.7) hold or $L$ is of the form in (i), (ii) or (iii) listed in the same lemma. Hence, by Corollary \ref{C:fYwhatA=Fm} it suffices to prove that $\fY_{\whatA}=\BF^m$ while (C1) holds together with each of (C2.1)--(C2.7), (i), (ii) and (iii).

For (i) and (ii) this is a consequence of Lemma \ref{L:Case (i) and (ii)}.

Since the matrix $L_0$ is in the span of $L_1,\ldots,L_m$ we can write $L_0=\sum_{k=1}^m \al_k L_k$.

Assuming $\sum_{k=1}^m \al_k \neq 1$, it is shown that $\fY_{\whatA}=\BF^m$ when one of (C2.1)--(C2.7) holds in Lemma \ref{L:SumAl_k=1}, while for $L$ as in (iii) it is shown that $\fY_{\whatA}=\BF^m$ in Lemma \ref{L:Case (iii) inv} for $\BF=\BC$. For case (iii) with $\BF=\BR$, $\fY_{\whatA}=\BF^m$ need not hold, but nonetheless     positivity and complete positivity still coincide by Lemma \ref{L:Case (iii) inv2}.

Assuming $\sum_{k=1}^m \al_k = 1$, it is proved that $\fY_{\whatA}=\BF^m$ when (C2.1) or (C2.2) holds in Lemma \ref{L:C2.1&C2.2}, when (C2.5) or (C2.6) holds in Lemma \ref{L:C2.5&C2.6}, when (C2.3) or (C2.4) holds in Lemma \ref{L:C2.3&C2.4}, when (C2.7) holds in Lemma \ref{L:C2.7} and when $L$ is as in (iii) in Lemma \ref{L:Case (iii) noninv}.

Together, these results prove all possible cases, by Lemma \ref{L:fE2=Fm}, so that we obtain that $\fY_{\whatA}=\BF^m$ holds whenever condition (C1) in \eqref{C1intro} is satisfied, except for the case that is covered in Lemma \ref{L:Case (iii) inv2}.

\subsection{All independents $L_1,\ldots,L_m$ in one row or column}

In Theorem \ref{T:Main} it is required that all matrices other than those selected as $L_1,\ldots,L_m$ are equal to a single matrix $L_0$. Whether this condition can be removed in general remains unclear. However, in the case that $L_1,\ldots,L_m$ are all contained in a single block row or block column, this condition can be removed.

\begin{proposition}\label{P:IndepsRowColumn}
Let $\cL$ as in \eqref{cL} be a $*$-linear matrix map with matricization $L$ and Choi matrix $\BL$. Set $m=\rank \BL$. Assume one can choose linearly independent $L_1,\ldots,L_m\in\BF^{n \times q}$ among the block entries of $L$, say $L_k=L_{i_kj_k}$ for $k=1,\ldots m$, in such a way that $i_k=i_l$ for all $k$ and $l$ or such that $j_k=j_l$ for all $k$ and $l$. In that case $\fY_{\whatA}=\BF^m$ and hence $\cL$ is completely positive in case $\cL$ is positive.
\end{proposition}

\begin{proof}[\bf Proof]
Assume $i_1=\cdots = i_m=r\in{1,\ldots,n}$. Then $j_k\neq j_l$ for $k\neq l$ and $m \leq q$. Moreover, in $\whatA$ the $(j
_k-1)n + r$-th column corresponds to the $k$-th unit vector $e_k$ in $\BF^m$. Let $y\in \BF^m$. Then take $x=e_r$ and $z\in \BF^q$ with $z_{i_k}=y_k$ for $k=1,\ldots,m$ and all other entries of $z$ equal to 0. It then follows that $\whatA(z \otimes x)=y$. Hence $\fY_{\whatA}=\BF^m$, as claimed. A similar argument applies when $j_k=j_l$ for all $k$ and $l$.
\end{proof}

As a special case of the above result one can consider the case where $L$ is an analytic or anti-analytic block Toeplitz matrix.

\begin{example}
Consider the case where the $*$-linear map $\BL$ has an anti-analytic block Toeplitz matrix structure:
\[
L=\mat{L_1&L_2&\cdots&L_{n}\\
0&L_1&\cdots&L_{n-1}\\
\vdots&\ddots&\ddots&\vdots\\
0&\cdots&0&L_1}.
\]
In this case one can take as independent matrices $L_1,\dots,L_n$ (leaving out some if there is a linear dependency) which are all in the first block row, so that Proposition \ref{P:IndepsRowColumn} applies and positivity and complete positivity coincide.
\end{example}

\section{The $2 \times 2$ case}\label{S:2x2}

In this section we consider all possible cases that can occur when $n=q=2.$ Then $m$ can take values $1,2,3$ and $4$.\smallskip

\paragraph{\bf m=1} Here $L$ can only take the form $L=\sbm{L_1&L_1\\L_1&L_1}$ with $L_1=\sbm{\ell_{11}&\ell_{11}\\\ell_{11}&\ell_{11}}$. Clearly (C1) in \eqref{C1} is met as well as conditions (C2.1) and (C2.2) of Lemma \ref{L:fE2=Fm}, and hence positivity and complete positivity of $\cL$ coincide by Lemmas \ref{L:SumAl_k=1} and \ref{L:C2.1&C2.2} and Corollary \ref{C:fYwhatA=Fm}.\smallskip

\paragraph{\bf m=2} In this case, in addition to the selected $L_1$ and $L_2$ there are two entries left, say $L_0$ and $L_0'$, that need not be the same. The four cases for $L$ of the form
\begin{equation}\label{2x2 with m=2}
\begin{bmatrix}L_1 & L_2 \\ L_0 & L_0' \end{bmatrix},\quad
\begin{bmatrix}L_0 & L_0' \\ L_1 & L_2 \end{bmatrix},\quad
\begin{bmatrix}L_1 & L_0 \\ L_2 & L_0' \end{bmatrix},\quad
\begin{bmatrix}L_0 & L_1 \\ L_0' & L_2 \end{bmatrix},
\end{equation}
are all covered by Proposition \ref{P:IndepsRowColumn}, which does not have the restriction that $L_0$ and $L_0'$ should be the same. The remaining cases
\[
\begin{bmatrix}L_1 & L_0 \\ L_0' & L_2 \end{bmatrix} \ands
\begin{bmatrix}L_0 & L_1 \\ L_2 & L_0' \end{bmatrix}
\]
are covered by Theorem \ref{T:Main} in case $L_0=L_0'$. For $L_0\neq L_0'$, since $L_1$ and $L_2$ are linearly independent either $L_0=0$, $L_1$ and $L_0$ are linearly independent or $L_2$ and $L_0$ are linearly independent. In the latter two cases we can make a different selection so that we are in one of the cases in \eqref{2x2 with m=2}. Hence, we only need to consider the case $L_0=0$. By a similar argument we can restrict to $L_0'=0$, but then $L_0=L_0'$ which is covered by Theorem \ref{T:Main}. Hence, in all cases with $m=2$ positivity and complete positivity of $\cL$ coincide.\smallskip

\paragraph{\bf m=3} Here we have four cases for $L$:
\begin{equation*}\label{2x2 with m=3}
\begin{bmatrix}L_1 & L_2 \\ L_3 & L_0 \end{bmatrix},\quad
\begin{bmatrix}L_1 & L_2 \\ L_0 & L_3 \end{bmatrix},\quad
\begin{bmatrix}L_1 & L_0 \\ L_2 & L_3 \end{bmatrix},\quad
\begin{bmatrix}L_0 & L_1 \\ L_2 & L_3 \end{bmatrix},
\end{equation*}
with only a single matrix that is dependent on the other entries. In neither of these cases (C1) holds. Keeping Remark \ref{R:Reorder} in mind, we may reorder block rows and block columns in $L$, so that one only has to deal with one of these four cases. The case where $L_0=0$ is covered in Example \ref{E:2x2upper}, and in that case we obtain that positivity and complete positivity of $\cL$ coincide. However, for the case where $L_0\ne 0$, there exists a case, given by the example below, that proves we can find a positive map $\cL$ that is not completely positive, at least for $\BF=\BR$. Hence in that case, it is not automatically true that positivity and complete positivity coincide. For $\BF=\BC$, however, we do not know at this stage whether positivity and complete positivity coincide for $m=3$.

\begin{example}\label{E:2x2Toeplitz}
Consider the case where we have a $*$-linear map $\cL$ whose matricization $L \in\BF^{4 \times 4}$ is a $2 \times 2$ block Toeplitz operator
\[
L=\mat{L_1&L_2\\ L_3 & L_1}=\mat{a_1&b_1&a_2&b_2\\c_1&a_1&c_2&a_2\\a_3&b_3&a_1&b_1\\c_3&a_3&c_1&a_1}.
\]
Since $\cL$ is $*$-linear, the Choi matrix $\BL$ is selfadjoint, and given by
\[
\BL
=\mat{a_1&a_3&a_2&a_1\\c_1&c_3&c_2&c_1\\b_1&b_3&b_2&b_1\\a_1&a_3&a_2&a_1}
=\mat{a_1&\ov{c}_1&\ov{b}_1&\ov{a}_1\\c_1&c_3&\ov{b}_3&\ov{a}_3\\b_1&b_3&b_2&\ov{a}_2\\a_1&a_3&a_2&a_1}.
\]
Hence, $*$-linearity of $\cL$ is equivalent to the Toeplitz structure in the blocks of $L$ together with $a_1,b_2,c_3\in\BR$, $a_3=\ov{c}_1$, $a_2=\ov{b}_1$ and $c_2=\ov{b}_3$. Hence, $*$-linearity of $\cL$ corresponds to
\[
L= \mat{a_1&b_1&\ov{b}_1&b_2\\c_1&a_1&\ov{b}_3&\ov{b}_1\\\ov{c}_1&\ov{c}_2&a_1&b_1\\c_3&\ov{c}_1&c_1&a_1} \ands
\BL= \mat{a_1&\ov{c}_1&\ov{b}_1&a_1\\c_1&c_3&c_2&c_1\\b_1&\ov{c}_2&b_2&b_1\\a_1&\ov{c}_1&\ov{b}_1&a_1}
\]
with $a_1,b_2,c_3\in\BR$ and all others arbitrarily from $\BF$. In this case
\[
\BH=\mat{a_1&\ov{b}_1&\ov{c}_1\\b_1&b_2&\ov{c}_2\\c_1&c_2&c_3},\qquad
\whatA=\mat{1&0&0&1\\0&0&1&0\\0&1&0&0}
\]
and $\mat{0&1&1}^T\not\in \fY_{\whatA}$, so that $\fY_{\whatA}\neq \BF^3$. Indeed, note that
\[
\whatA (z \otimes x)=\mat{z_1x_1+z_2x_2\\ z_2 x_1\\ z_1 x_2},
\]
so that if $\mat{a&1&1}^T=\whatA (z \otimes x)$, then $z_2,\, x_1$ must be non-zero and of the same sign and $z_1,\, x_2$ must be non-zero and of the same sign, so that $z_1x_1,\, z_2x_2$ will be non-zero and of the same sign, and thus $a=z_1x_1 + z_2x_2$ cannot be 0.

We claim that in this case it can happen that $\cL$ is positive but not completely positive, at least for $\BF=\BR$. Consider the case where $b_1=c_1=c_3=0$, $a_1,b_2>0$ and $c_2=-2 a_1$. Then
\[
\BH=\mat{a_1&0&0\\0&b_2&-2 a_1\\0&-2 a_1&0} \not\geq 0,
\]
so that $\cL$ is not completely positive. Furthermore, we have
\[
\left(I_2 \otimes x\right)^*\BL\left(I_2 \otimes x\right) =\mat{a_1 |x_1|^2 & a_1 x_2 \ov{x}_1-2 a_1 x_1 \ov{x}_2\\ a_1 x_1 \ov{x}_2-2 a_1 x_2 \ov{x}_1& b_2 |x_1|^2  + a_1|x_2|^2}.
\]
For $x_1=0$, this matrix clearly is positive semidefinite. Hence assume $x_1\neq 0$. In that case the Schur complement of the above matrix with respect to the left upper corner is
\begin{align*}
& b_2 |x_1|^2  + a_1|x_2|^2 - \frac{|a_1 x_1 \ov{x}_2-2 a_1 x_2 \ov{x}_1|^2}{a_1 |x_1|^2}
=\\
& \qquad = b_2 |x_1|^2  + a_1|x_2|^2 -  \frac{1}{a_1} \left|a_1 \frac{x_1}{|x_1|} \ov{x}_2-2 a_1 x_2 \frac{\ov{x}_1}{|x_1|}\right|^2.
\end{align*}
For $\BF=\BC$ there are $x\in\BC^2$ where this number is negative, however, for $\BF=\BR$, it is equal to $b_2 |x_1|^2$ and hence always positive. Thus, for $\BF=\BR$, the above $2 \times 2$ block Toeplitz matrix gives another example of a $*$-linear positive map $\cL$ which is not completely positive, while for $\BF=\BC$ the map $\cL$ is not positive.
\end{example}

\paragraph{\bf m=4} This case is covered in Example \ref{E:2x2fullblock}, where it is shown that  positivity of $\cL$ need not imply complete positivity of $\cL.$\smallskip

\paragraph{\bf Acknowledgments}
This work is based on research supported in part by the National Research Foundation of South Africa (NRF) and the DSI-NRF Centre of Excellence in Mathematical and Statistical Sciences (CoE-MaSS). Any opinion, finding and conclusion or recommendation expressed in this material is that of the authors and the NRF and CoE-MaSS do not accept any liability in this regard.

\end{document}